\documentclass[12pt, a4paper]{amsart}

\newtheorem{theorem}{Theorem}[section]

\newtheorem{property}[theorem]{Property}
\newtheorem{lemma}[theorem]{Lemma}
\newtheorem*{pro-a}{Proposition a}
\newtheorem*{pro-b}{Proposition b}
\newtheorem*{pro-c}{Proposition c}
\newtheorem*{pro-d}{Proposition d}
\newtheorem*{thm-a}{Theorem A}
\newtheorem*{thm-c}{Theorem C}
\newtheorem*{algo-b}{Algorithm B}
\newtheorem*{algo-d}{Algorithm D}

\newcommand{\T}{\mathbb{T}}
\newcommand{\K}{\text{\rm Ker}}
\newcommand{\w}{\omega}
\newcommand{\s}{\sigma}
\newcommand{\A}{\alpha}

\title[The numbers of powers in the Tribonacci sequence]
{The numbers of powers in the Tribonacci sequence}

\author{Yu-Ke HUANG}
\address{School of Mathematics and Systems Science, Beihang University (BUAA), Beijing, 100191, P. R. China}
\email{huangyuke@buaa.edu.cn(Corresponding author).}

\author{Zhi-Ying WEN}
\address{Department of Mathematical Sciences, Tsinghua University, Beijing, 100084, P. R. China.}
\email{wenzy@tsinghua.edu.cn.}

\thanks{The research is supported by the Grant NSFC No.11431007, No.11271223 and No.11371210.}


\begin{document}

\begin{abstract}
The Tribonacci sequence $\T$ is the fixed point of the substitution $\s(a)=ab$, $\s(b)=ac$, $\s(c)=a$. The prefix of $\T$ of length $n$ is denoted by
$\T[1,n]$.
The main result is threefold, we give:
(1) explicit expressions of the numbers of distinct squares and cubes in $\T[1,n]$;
(2) algorithms for counting the numbers of repeated squares and cubes in $\T[1,n]$;
(3) a discussion about $\A$-powers in $\T[1,n]$ for $\A\geq2$ and $n\geq1$.
\end{abstract}

\maketitle

\section{Introduction}

A sequence is said to be \emph{recurrent} if every factor occurs infinitely often, such as Sturmian sequences and Arnoux-Rauzy sequences \cite{AS2003}.
We say a (finite or infinite) word $\tau$ contains an \emph{$\A$-power} (real $\A>1$) if $\tau$ has a factor of the form $\w^{\lfloor\A\rfloor}\w'$. Here $\lfloor\A\rfloor$ is the integer not more than $\A$, $\w$ is a factor of $\tau$ and $\w'$ is a prefix of $\w$.
The study of powers in recurrent sequences has a
long history, and among the many significant contributions, we mention \cite{DL2003,FS1999,FS2014,JP2002,MS2014,S2010,TW2007}.

Let $\mathcal{A}=\{a,b,c\}$ be a three-letter alphabet.
The Tribonacci substitution $\s$ is defined by $\s(a)=ab$, $\s(b)=ac$ and $\s(c)=a$, which is also called the Rauzy substitution.
The Tribonacci sequence $\T$ is an infinite sequence over
the alphabet $\mathcal{A}$, defined to be
the fixed point beginning with the letter $a$ of the Tribonacci substitution.
$\T$ is a natural generalization of the Fibonacci sequence, and also a classical example of Arnoux-Rauzy sequences, see \cite{AS2003,FBFMS2002}.
Thus $\T$ is recurrent, each factor $\w$ occurs infinitely many times.
Let $\w_p$ be the position of last letter (also called the ending position) of the $p$-th occurrence of $\w$ for $p\geq1$. For instance, $(ab)_2=6$.

In this paper we mainly consider integer powers occurring in $\T$.
Let $\w\prec\T$ be a factor of $\T$.
The word $\w\w$ (resp. $\w\w\w$)
is called \emph{square} (resp. \emph{cube}) if it is a factor of $\T$. In this case, $\A=2$ (resp. $\A=3$).
As we know, $\T$ contains no fourth powers, see Theorem 4 in \cite{MS2014} for instance.
The properties of squares and cubes are objects of a great interest in many aspects of mathematics and computer science.

Define $T_{-2}=\epsilon$ (empty word), $T_{-1}=c$
and $T_m=\s^m(a)$ for $m\geq0$, called the
$m$-th Tribonacci word.
We denote the length of $T_m$ (the number of letters it contains) by $|T_m|=t_m$ for $m\geq-2$, called the $m$-th Tribonacci number. We define $t_{-3}=0$.

The first few Tribonacci words and numbers are listed below:
\begin{center}
\begin{tabular}{|c|c|c|c|c|c|c|c|c|}
\hline
$m$&-3&-2&-1&0&1&2&3&4\\\hline
$T_{m}$&&$\epsilon$&$c$&$a$&$ab$&$abac$&$abacaba$&$abacabaabacab$\\
\hline
$t_{m}$&0&0&1&1&2&4&7&13\\
\hline
\end{tabular}
\end{center}

\smallskip

Let $\T[1,n]$ be the prefix of $\T$ of length $n$.
The notation $\nu\triangleright\w$ means that the word $\nu$ is a suffix of the word $\w$.
In this paper, we mainly consider the four functions below, for $n\geq0$:
\begin{equation*}
\begin{cases}
A(n)=&\#\{\w\mid\w\w\prec\T[1,n]\},\\
&\text{the number of distinct squares in }\T[1,n];\\
B(n)=&\#\{(\w,p)\mid \w\w\triangleright\T[1,\w_{p+1}],\ \w_{p+1}\leq n\},\\
&\text{the number of repeated squares in }\T[1,n];\\
C(n)=&\#\{\w \mid \w\w\w\prec\T[1,n]\},\\
&\text{the number of distinct cubes in }\T[1,n];\\
D(n)=&\#\{(\w,p)\mid \w\w\w\triangleright\T[1,\w_{p+2}],\ \w_{p+2}\leq n\},\\
&\text{the number of repeated cubes in }\T[1,n].
\end{cases}
\end{equation*}

The methods for counting the four functions have attracted many authors.
In 2006, Glen \cite{G2006} gave expressions of $A(t_m)$.
In 2014, Mousavi-Shallit \cite{MS2014} gave expressions of $B(t_m)$ and $D(t_m)$.
In this paper, we give:
(1) explicit expressions of $A(n)$ and $C(n)$;
(2) algorithms for counting $B(n)$ and $D(n)$ for all $n\geq1$.
Moreover, we give a discussion about $\A$-powers in $\T[1,n]$ for $\A\geq2$.

\medskip

This paper is organized as follows.
The main results are presented in Section 2. Section 3 is devoted to explain the main difficulties and how we overcome them. We count the number of squares and cubes in Sections 4 and 5, respectively.
In the two sections, we establish recursive structures called square trees and cube trees.
In Section 6, we discuss the enumeration problem for $\A$-powers where $\A\geq2$.

\section{Main Results}

In order to calculate $A(n)$, $B(n)$, $C(n)$ and $D(n)$, we only need to calculate the four difference functions below, for $n\geq1$:
\begin{equation*}
\begin{cases}
a(n)=A(n)-A(n-1),&b(n)=B(n)-B(n-1),\\
c(n)=C(n)-C(n-1),&d(n)=D(n)-D(n-1).
\end{cases}
\end{equation*}
Obviously, $A(0)=B(0)=C(0)=D(0)=0$. Indeed, this is done in Propositions a-d.
Let $[m..n]$ denote the set of integers $\{m,m+1,\ldots,n\}$ for $m\leq n$.
Note that $[m]=\{m\}$.

\begin{pro-a}[] For each positive integer $n$, the number $a(n)=1$ if
\begin{equation*}
\begin{split}
n&\in\bigcup_{m\geq 3}\big([2t_{m-1}..t_{m}+2t_{m-3}-1]
\cup[2t_{m}-t_{m-1}..\tfrac{3t_{m}+t_{m-2}-3}{2}]\big) \\
& =  [8]\cup [10]\cup[14,15,16]\cup[19,20]\cup \cdots
\end{split}
\end{equation*}
Otherwise, $a(n)=0$.
\end{pro-a}

For $U=[m..n]$ and $V=[n+1..h]$, we denote $U\cup V=[m..h]$ by $[U,V]$.
For integer $k$, denote $U\oplus k=[m+k..n+k]$.
Three infinite sequences of sets $\{U_m\}_{m\geq0}$,
$\{V_m\}_{m\geq0}$ and $\{W_m\}_{m\geq0}$ are defined as below.
\begin{equation*}
\begin{cases}
U_m\ \,=\ [1..t_{m-3}]\oplus\frac{t_{m+1}+t_{m-1}-3}{2},\\
V_m\ \ =\ [1..t_{m-2}]\oplus\frac{3t_{m}-t_{m-2}-3}{2},\\
W_m\ =\ [1..t_{m-1}]\oplus\frac{3t_{m}+t_{m-2}-3}{2}.
\end{cases}
\end{equation*}
Here $U_0$, $V_0$, $U_1$ are empty sets.
$W_0=[1]$, $V_1=[2]$, $W_1=[3]$,
$U_2=[4]$, $V_2=[5]$, $W_2=[6,7]$. Obviously the union of
$U_0$, $V_0$, $W_0$, $U_1$, $V_1$, $W_1$, $U_2$, $V_2$, $W_2$, $\ldots$ is $\mathbb{Z}^{+}$ (the positive integers).

\smallskip

We denote the vector $[b(m),b(m+1),\ldots,b(n)]$ by $b([m..n])$ for short.

For vectors $U=[x_1,x_2,\ldots,x_m]$ and $V=[y_1,y_2,\ldots,y_n]$, $[U,V]=[x_1,x_2,\ldots,x_m,y_1,y_2,\ldots,y_n]$. When $m=n$, we denote $U+V=[x_1+y_1,x_2+y_2,\ldots,x_m+y_m]$.

\begin{pro-b}[]
We have $b(n)=0$ for $n\leq7$, and for $m\geq3$
\begin{equation*}
\begin{cases}
b(U_{m+2})
=b\big([U_{m-1},V_{m-1},W_{m-1}]\big)
+[\underbrace{1,\ldots,1,}_{\frac{-t_{m+1}+5t_{m-1}+1}{2}}
\underbrace{0,\ldots,0}_{\frac{t_{m+1}-3t_{m-1}-1}{2}}];\\
b(V_{m+1})
=b\big([U_{m-1},V_{m-1},W_{m-1}]\big)
+[\underbrace{0,\ldots,0,}_{\frac{t_{m-1}-t_{m-3}+1}{2}}
\underbrace{1,\ldots,1}_{\frac{t_{m-1}+t_{m-3}-1}{2}}];\\
b(W_m)\kern 0.25cm
=b\big([U_{m-1},V_{m-1},W_{m-1}]\big)
+[\underbrace{0,\ldots,0,}_{\frac{t_{m-1}+t_{m-3}+1}{2}}
\underbrace{1,\ldots,1}_{\frac{t_{m-1}-t_{m-3}-1}{2}}].
\end{cases}
\end{equation*}
Here the number under `` $\stackrel{\underbrace{}}{}$" means the number of repeated elements.
\end{pro-b}

\noindent\emph{Example.}
$b(U_3)=b([8])=[1],\kern 0.75cm b(V_3)=b([9,10])=[0,1]$,

$b(W_3)=b([11..14])=[0,0,0,1]$,

$b(U_4)=b([15,16])=[1,1],\kern 1cm b(V_4)=b([17..20])=[0,0,1,1]$,

$b(W_4)=b([21..27])=[1,0,1,0,0,1,2]$.

\begin{pro-c}[] For each positive integer $n$, the number $c(n)=1$ if
\begin{equation*}
\begin{split}
n&\in\bigcup_{m\geq7}[t_{m-1}+2t_{m-4}..\tfrac{3t_{m-1}-t_{m-3}-3}{2}]\\
&=[58]\cup[107,108]\cup[197..200] \cup \cdots
\end{split}
\end{equation*}
Otherwise, $c(n)=0$.
\end{pro-c}

We define an infinite sequence of sets $\{\Theta_m\}_{m\geq0}$, where
$$\Theta_m=[U_{m},V_{m},W_{m}]=[1..t_{m}]\oplus\tfrac{t_{m+1}+t_{m-1}-3}{2}.$$
For instance, $\Theta_0=[1]$, $\Theta_1=[2,3]$, $\Theta_2=[4,5,6,7]$, $\Theta_{3}=[8..14]$.

\begin{pro-d}[] We have $d(n)=0$ for $n\leq51$, and for $m\geq 6$,
\begin{equation*}
\begin{split}
d(\Theta_{m})
=&~d\big([\Theta_{m-3},\Theta_{m-2},\Theta_{m-1}]\big)\\
&+[\underbrace{0,\ldots,0,}_{\tfrac{-t_{m-1}+5t_{m-3}+1}{2}}
\underbrace{1,\ldots,1,}_{\tfrac{t_{m-1}-3t_{m-3}-1}{2}}
\underbrace{0,\ldots,0}_{t_{m-1}+t_{m-2}}].
\end{split}
\end{equation*}
\end{pro-d}

\noindent\emph{Example.}
$d(\Theta_{6})=d([52..95])=[\underbrace{0,\ldots,0}_6,1,\underbrace{0,\ldots,0}_{37}]$;

$d(\Theta_{7})=d([96..176])
=[\underbrace{0,\ldots,0}_{11},1,1,\underbrace{0,\ldots,0}_{30},1,\underbrace{0,\ldots,0}_{37}]$.

\medskip

By $A(n)=\sum_{i=1}^n a(i)$ and $C(n)=\sum_{i=1}^n c(i)$,
Propositions a and~c imply
the explicit expressions of $A(n)$ and $C(n)$,
as stated in Theorems A and C.
For $m\geq3$, we define $\A_m=2t_{m-1}$, $\beta_m=t_{m}+2t_{m-3}-1$, $\gamma_m=2t_{m}-t_{m-1}$ and $\theta_m=\frac{3t_{m}+t_{m-2}-3}{2}$.

\begin{thm-a}[The numbers of distinct squares, $A(n)$]\

For $n\leq7$, $A(n)=0$. For $n\geq8$, let $m$ be the positive integer such that $\A_m\leq n<\A_{m+1}$, then $m\geq3$,
\begin{equation*}
A(n)=
\begin{cases}
n-\frac{1}{2}(t_{m}+t_{m-3}+m+3),&\A_m\leq n<\beta_{m};\\
\frac{1}{2}(t_{m}+3t_{m-3}-m-5),&\beta_{m}\leq n<\gamma_m;\\
n-\frac{1}{2}(t_{m-1}+3t_{m-2}+m+3),&\gamma_m\leq n<\theta_m;\\
\frac{1}{2}(2t_{m-1}+t_{m-2}+3t_{m-3}-m-6),&\text{otherwise}.
\end{cases}
\end{equation*}
\end{thm-a}

\noindent\emph{Remark.} For $m\geq3$, $\theta_{m-1}\leq t_m<\A_{m}$. Thus by Theorem A we have $A(t_m)=A(\theta_{m-1})=\frac{1}{2}(2t_{m-2}+t_{m-3}+3t_{m-4}-m-5)$.
This is different but equivalent to Theorem 6.30 in Glen \cite{G2006}.

\smallskip

The first values of $a(n)$ and $A(n)$ are listed in the following table:
\begin{center}
\begin{tabular}{|c|c|c|c|c|c|c|c|c|c|c|c|}
\hline
$n$&[1..7]&8&9&10&[11..13]&14&15&16&17,18&19&20\\\hline
$a(n)$&0&1&0&1 &0    &1 &1 &1 &0    &1 &1\\\hline
$A(n)$&0&1&1&2 &2    &3 &4 &5 &5    &6 &7\\\hline
\end{tabular}
\end{center}

\begin{thm-c}[The numbers of distinct cubes, $C(n)$]\

For $n\leq 57$, $C(n)=0$. For $n\geq 58$, let $m$
be the positive integer such that $t_{m-1}+2t_{m-4}\leq n<t_{m}+2t_{m-3}$, then $m\geq7$,
\begin{equation*}
C(n)=
\begin{cases}
n-\frac{1}{2}(3t_{m-2}+t_{m-3}+4t_{m-4}+m-6),&n\leq\frac{3t_{m-1}-t_{m-3}-3}{2};\\
\frac{1}{2}(t_{m-5}+t_{m-6}-m+3),&otherwise.
\end{cases}
\end{equation*}
\end{thm-c}

\noindent\emph{Remark.}
For $m\geq7$,
$t_m>\frac{3t_{m-1}-t_{m-3}-3}{2}$,
$C(t_m)=\frac{1}{2}(t_{m-5}+t_{m-6}-m+3)$.
When $m\leq6$, $C(t_m)=0$. This expression holds for $m\in[3..6]$ too.

\medskip

The first values of $c(n)$ and $C(n)$ are listed in the following table:
\begin{center}
\begin{tabular}{|c|c|c|c|c|c|c|c|c|c|c|}
\hline
$n$&[1..57]&58&[59..106]&107&108&[109..196]&197&198&199\\\hline
$c(n)$&0&1 &0     &1  &1  &0      &1  &1  &1\\\hline
$C(n)$&0&1 &1     &2  &3  &3      &4  &5  &6\\\hline
\end{tabular}
\end{center}

\medskip

The maximal element in $W_m$ (resp. $\Theta_m$) is $\tfrac{t_{m+2}+t_{m}-3}{2}$, which is larger than $2t_m$ for $m\geq4$.
Since $B(n)=\sum_{i=1}^n b(i)$ and $D(n)=\sum_{i=1}^n d(i)$, we can calculate $B(n)$ and $D(n)$ by Algorithms B. and D.

\begin{algo-b}[The numbers of repeated squares, $B(n)$]\

1. Find the positive integer $M$ such that $2t_{M-1}< n\leq2t_M$;

2. Let $\vec{b}=[\underbrace{0,\ldots,0}_7,1,0,1,0,0,0,1]$. (For instance, $\vec{b}[8]=1$.)

\quad\ For $n\geq15$, i.e., $M\geq4$, go to Step 3; otherwise go to Step 5.

3. For $m\geq4$, calculate $b(U_m)$, $b(V_m)$ and $b(W_m)$ by Proposition b;

\quad\ replace the vector $\vec{b}$ by $[\vec{b},b(U_m),b(V_m),b(W_m)]$; replace $m$ by $m+1$.

4. Repeat Step 3, until $m=M+1$.

5. $B(n)$ is equal to the sum of the first $n$ elements in vector $\vec{b}$.
\end{algo-b}

\noindent\emph{Example.} We can calculate $B(n)$ quickly by Algorithm B. For instance, we get both $B(t_{24})=15,\!795,\!657$ and $B(5,\!000,\!000)=32,\!561,\!325$ within 2 seconds by computer.
We can check the former by Theorem 21 in \cite{MS2014}. Notice that the notation ``$t_m$" in this paper is written by ``$T_{n+2}$" in \cite{MS2014}.

\begin{algo-d}[The numbers of repeated cubes, $D(n)$]\

1. Find the positive integer $M$ such that $2t_{M-1}< n\leq2t_M$;

2. Let $\vec{d}=[\underbrace{0,\ldots,0}_{57},1,
\underbrace{0,\ldots,0}_{48},1,1,
\underbrace{0,\ldots,0}_{30},1,\underbrace{0,\ldots,0}_{37}]$.

\quad\ For $n\geq177$, i.e., $M\geq8$, go to Step 3; otherwise go to Step 5.

3. For $m\geq8$, calculate $d(\Theta_m)$ by Proposition d;

\quad\ replace the vector $\vec{d}$ by $[\vec{d},d(\Theta_m)]$; replace $m$ by $m+1$.

4. Repeat Step 3, until $m=M+1$.

5. $D(n)$ is equal to the sum of the first $n$ elements in vector $\vec{d}$.
\end{algo-d}

\noindent\emph{Example.} We get both $D(t_{24})=819,\!290$ and $D(5,\!000,\!000)=1,\!699,\!525$ within 2 seconds by computer.

\section{Main Difficulties and main tools}

Taking ``square" for example, the main difficulty in counting the numbers of squares in the Tribonacci sequence is twofold: D1 and D2.

\smallskip

\textbf{D1.} The positions of all squares are not easy to be determined.
As we know, the known results mainly consider the enumeration problem in the prefix of sequence of special length, such as $\T[1,t_m]=T_m=\s^m(a)$.
In our opinion, all these methods are hardly generalized to solve the enumeration problem in arbitrary prefix.

\smallskip

We overcome this difficulty by using the ``gap sequence" property of $\T$, which we introduced and studied in \cite{HW2015-2}.

Let $\tau=x_1x_2\cdots x_n$ be a finite word, $x_i\in\mathcal{A}$.
For any $1\leq i\leq j\leq n$,
we define $\tau[i]=x_i$, $\tau[i,j]=x_ix_{i+1}\cdots x_{j-1}x_j$.
Note that $\tau[i,i]=x_i$ and $\tau[i,i-1]=\epsilon$.
For $3\leq j+2\leq i\leq n$, $\tau[i,j]$ is a inverse word $(x_{j+1}x_{j+2}\cdots x_{i-1})^{-1}$.
Let $\w$ be a factor of $\T$.
The $p$-th \emph{gap} word of $\w$ is
$G_{p}(\omega)=\T[\w_p+1,\w_{p+1}-|\w|]$ for $p\geq1$.
When the $p$-th and $(p+1)$-th occurrences of $\w$ are adjacent (resp. overlapped), $G_p(\w)$ is empty word (resp. an inverse word). For instance, let $\w=abacaba$, then $G_1(\w)=\epsilon$ and $G_2(\w)=a^{-1}$.
The sequence $\{G_p(\omega)\}_{p\geq1}$ is called
the \emph{gap sequence} of factor $\w$.
By the gap sequence property of $\T$, the set $\{\w\mid \w_{p+i}=\w_{p}+|\w|,i\geq2,p\geq1\}$ is empty. Thus the word $\w\w$ is a square means at least one element in set $\{G_p(\w)\mid p\geq1\}=\{G_1(\w),G_2(\w),G_4(\w)\}$ is empty word $\epsilon$.

\begin{property}[Theorem 3.3 in \cite{HW2015-2}]\label{Gt}
For each factor $\w\prec\T$, the gap sequence $\{G_p(\w)\}_{p\geq1}$ is itself still a Tribonacci sequence over the alphabet $\{G_1(\w),G_2(\w),G_4(\w)\}$.
\end{property}

\noindent\emph{Example.} Consider $ab\prec\T$, $G_1(ab)$ (resp. $G_2(ab)$, $G_4(ab)$) is factor $ac$ (resp. $a$, $\epsilon$), which is denoted by $A$  (resp. $B$, $C$) below.
\begin{equation*}
\begin{split}
\T=~&ab\underbrace{ac}_Aab\underbrace{a}_B ab\underbrace{ac}_Aab\underbrace{\epsilon}_C
ab\underbrace{ac}_Aab\underbrace{a}_Bab\underbrace{ac}_A\\
&ab\underbrace{ac}_Aab\underbrace{a}_B
ab\underbrace{ac}_Aab\underbrace{\epsilon}_Cab\underbrace{ac}_Aab\underbrace{a}_B \quad\cdots
\end{split}
\end{equation*}

The other important tool is ``kernel word".
We define the kernel numbers that $k_{0}=0$, $k_{1}=k_{2}=1$, $k_m=k_{m-1}+k_{m-2}+k_{m-3}-1$ for $m\geq3$.
Then $k_m=\frac{t_{m-3}+t_{m-5}+1}{2}$ for $m\geq3$.
The \emph{kernel} word with order $m$ is defined as
$K_1=a$, $K_2=b$, $K_3=c$, $K_m=\delta_{m-1}T_{m-3}[1,k_m-1]$ for $m\geq4$, where $\delta_m$ is the last letter of $T_m$.

The first few kernel words and numbers are listed below:
\begin{center}
\begin{tabular}{|c|c|c|c|c|c|c|c|}
\hline
$m$&1&2&3&4&5&6&7\\\hline
$K_{m}$&$a$&$b$&$c$&$aa$&$bab$&$cabac$&$aabacabaa$\\
\hline
$k_{m}$&1&1&1&2&3&5&9\\
\hline
\end{tabular}
\end{center}

\smallskip

Let $\K(\w)$ be the maximal kernel word occurring in factor $\omega$. For instance $\K(abacab)=c$.
By Theorem 4.3 in \cite{HW2015-2}, $\K(\w)$ occurs in $\w$ only once.
Moreover,

\begin{property}[Theorem 4.11 in \cite{HW2015-2}]\label{wpt}
Let $\w\prec \T$ be a factor. For all $p\geq 1$, the difference
$\w_p-\K(\w)_p= \w_1-\K(\w)_1$ is independent of $p$.
\end{property}

\noindent\emph{Example.} $\K(aba)=b$, $(aba)_1=3$, $b_1=2$, $(aba)_6=20$ and $b_6=19$.

\medskip

\textbf{D2.} By the gap sequence property of $\mathbb{T}$, we can find out all distinct squares in $\mathbb{T}[1,n]$. An immediate idea is counting the number of occurrences of each square, and the summation of them are the numbers of repeated squares in $\mathbb{T}[1,n]$. But this method is complicated.

\smallskip

We overcome this difficulty by studying the relations among the positions $\omega_p$ for all $\omega\prec\T$ and $p\geq1$. Then we establish a recursive structure, called square trees, see Figure \ref{Fig:1}.
In Section 6, we generalize the square trees to $\A$-power trees.
We also show the evolution of $\A$-power trees according to the increase of parameter $\A$ from 2 to 3.

\section{The number of squares}

In this section, we count the number of distinct and repeated squares in the Tribonacci sequence. First we give some basic properties of squares.
By Lemma 4.7, Definition 4.12 and Corollary 4.13 in \cite{HW2015-2}, any factor $\omega$ with kernel $K_m$ can be expressed uniquely as
\begin{equation}\label{E1}
\omega=T_{m-1}[i,t_{m-1}-1]K_mT_{m}[k_m,k_m+j-1],
\end{equation}
where $1\leq i\leq t_{m-1}$, $0\leq j\leq t_{m-1}-1$.
Then $\w\w\prec\T$ is equivalent to
$$\w\w=T_{m-1}[i,t_{m-1}-1]K_m\ G(K_m)\ K_mT_{m}[k_m,k_m+j-1],$$
where $G(K_m)=T_{m}[k_m,k_m+j-1]T_{m-1}[i,t_{m-1}-1]$ is the gap between two consecutive $K_m$,
i.e., $G(K_m)=\T[(K_m)_p+1,(K_m)_{p+1}-k_m]$ for some $p$.
Moreover, $|G(K_m)|=t_{m-1}-i+j$ and $|\omega|=|G(K_m)|+k_m$.

By Proposition 3.2 in \cite{HW2015-2}, $G(K_m)$ has three cases
\begin{equation*}
\begin{cases}
G_1=T_{m-2}[k_{m},t_{m-2}]T_{m-3}T_{m-1}[1,t_{m-1}-1],\ |G_1|=t_{m}-k_{m};\\
G_2=T_{m-2}[k_{m},t_{m-2}]T_{m-1}[1,t_{m-1}-1],\ |G_2|=t_{m-2}+t_{m-1}-k_{m};\\
G_4=T_{m-1}[k_{m},t_{m-1}-1],\ |G_4|=t_{m-1}-k_{m}.
\end{cases}
\end{equation*}

\textbf{Case 1.} $G(K_m)=G_1$. Since $|G(K_m)|=|G_1|$, $j=t_{m-2}+t_{m-3}-k_{m}+i$. So $0\leq j\leq t_{m-1}-1$ gives a range of $i$. Comparing this range with $1\leq i\leq t_{m-1}$, we have $1\leq i\leq k_{m+1}-1$ and $m\geq3$.
Furthermore,
\begin{equation}\label{E2}
\w\w=T_{m}[i,t_{m}-1]\underbrace{\delta_{m}T_{m-1}[1,k_{m+1}-1]}_{\text{This is }K_{m+1}}T_{m+1}[k_{m+1},t_{m}+i-1].
\end{equation}
This means $K_{m+1}\prec\w\w$.
Moreover, $K_{m+2}$, $K_{m+3}$ and $K_{m+4}$ are not the factors of $\w\w$. Since $|\w\w|<k_{m+5}$, kernel word $K_{m+h}$ for $h\geq5$ are not the factors of $\w\w$.
Thus $K_{m+1}$ is the maximal kernel word in $\omega\omega$, i.e. $\K(\omega\omega)=K_{m+1}$ for $m\geq3$.
In this case, $|\omega|=|G(K_m)|+k_m=t_{m}$.

\smallskip

By analogous arguments, we have

\textbf{Case 2.} $G(K_m)=G_2$. In this case, $1\leq i\leq k_{m+2}-1$, $m\geq2$.
\begin{equation}\label{E3}
\w\w
=T_{m}[i,t_{m-1}+t_{m-2}-1]\underline{K_{m+2}}T_{m+1}[k_{m+2},t_{m-1}+t_{m-2}+i-1].
\end{equation}
So $\K(\omega\omega)=K_{m+2}$ and $|\omega|=t_{m-2}+t_{m-1}$.

\textbf{Case 3.} $G(K_m)=G_4$. In this case, $k_{m}\leq i\leq t_{m-1}$, $m\geq1$.
\begin{equation}\label{E4}
\w\w
=T_{m-1}[i,t_{m-1}-1]\underline{K_{m+3}}T_{m+1}[k_{m+3},t_{m-1}+i-1].
\end{equation}
So $\K(\omega\omega)=K_{m+3}$ and $|\omega|=t_{m-1}$.

\medskip

\noindent\emph{Remark.}
By the three cases of squares, we have: (1) all squares in $\mathbb{T}$ are of length $2t_m$ or $2t_m+2t_{m-1}$ for some $m\geq0$; (2) for all $m\geq0$, there exists a square of length $2t_m$ and $2t_m+2t_{m-1}$ in $\mathbb{T}$. They are known results of Mousavi-Shallit, see Theorem 5 in \cite{MS2014}.

\medskip

We denote by $|\w|_\gamma$ the number of letter $\gamma$ in $\w$ where $\gamma\in\mathcal{A}$.

Theorem 6.1 and Remark 6.2 in \cite{HW2015-2} gave the positions of all occurrences of $K_m$. An equivalent property is

\begin{property}[]\label{P} For $m,p\geq1$,
the ending position of the $p$-th occurrence of $K_m$, denoted by $(K_m)_p$, is equal to
$$pt_{m-1}+|\T[1,p-1]|_a(t_{m-2}+t_{m-3})+|\T[1,p-1]|_bt_{m-2}+k_m-1.$$
In particular, $(K_m)_1=k_{m+3}-1=\frac{t_{m}+t_{m-2}-1}{2}$ for $m\geq1$.
$a_p=p+|\mathbb{T}[1,p-1]|_a+|\mathbb{T}[1,p-1]|_b$,
$b_p=2p+2|\mathbb{T}[1,p-1]|_a+|\mathbb{T}[1,p-1]|_b$,
$c_p=4p+3|\mathbb{T}[1,p-1]|_a+2|\mathbb{T}[1,p-1]|_b$ for $p\geq1$.
\end{property}

We define three sets for $m\geq4$ and $p\geq1$,
\begin{equation}\label{N1}
\begin{cases}
K^1_{m,p}\!=\!\{(\w\w)_p\mid \K(\w\w)\!=\!K_m,|\w|\!=\!t_{m-1},\w\w\prec\T\};\\
K^2_{m,p}\!=\!\{(\w\w)_p\mid \K(\w\w)\!=\!K_m,|\w|\!=\!t_{m-3}+t_{m-4},\w\w\prec\T\};\\
K^3_{m,p}\!=\!\{(\w\w)_p\mid \K(\w\w)\!=\!K_m,|\w|\!=\!t_{m-4},\w\w\prec\T\}.
\end{cases}
\end{equation}
Obviously they correspond all ending positions of the three cases of squares, respectively.
The following property is a consequence of Expressions (\ref{E2})-(\ref{E4}) and Property \ref{P}.

\begin{property}\label{p4.3}
For $m\geq4$ and $p\geq1$,
\begin{equation}\label{E5}
\begin{cases}
K^1_{m,p}
=[\frac{t_{m-1}-t_{m-3}+1}{2}..t_{m-2}-1]\oplus\big((K_m)_p+t_{m-3}+t_{m-4}\big);\\
K^2_{m,p}
=[\frac{t_{m-1}-2t_{m-2}+t_{m-3}+1}{2}..t_{m-3}-1]\oplus\big((K_m)_p+t_{m-4}\big);\\
K^3_{m,p}
=[0..\frac{-t_{m-2}+5t_{m-4}-1}{2}]\oplus(K_m)_p.
\end{cases}
\end{equation}
Moreover $\#K^1_{m,p}=\#K^2_{m,p}=\frac{t_{m+1}-3t_{m-1}-1}{2}$, $\#K^3_{m,p}=\frac{-t_{m-2}+5t_{m-4}+1}{2}$.
\end{property}

\subsection{Proofs of Proposition a. and Theorem A}

Now we consider the number of distinct squares in $\T[1,n]$.
Take $p=1$ in Expression~(\ref{E5}). Since $(K_m)_1=\frac{t_{m}+t_{m-2}-1}{2}$, we have
\begin{equation*}
\begin{cases}
K^1_{m,1}=[2t_{m-1}..\frac{t_{m+1}+t_{m-1}-3}{2}];\\
K^2_{m,1}=[2t_{m-1}-t_{m-2}..\frac{3t_{m-1}+t_{m-3}-3}{2}];\\
K^3_{m,1}=[\frac{t_{m}+t_{m-2}-1}{2}..t_{m-1}+2t_{m-4}-1].
\end{cases}
\end{equation*}
It is easy to check that all these sets are pairwise disjoint.
Moreover $\max K^1_{m,1}+1=\min K^3_{m+1,1}$,
$[K^1_{m,1},K^3_{m+1,1}]=[2t_{m-1}..t_{m}+2t_{m-3}-1]$.
Therefore we get a chain in increasing order
$$K^3_{4,1},K^2_{4,1},[K^1_{4,1},K^3_{5,1}],\ldots,
K^2_{m,1},[K^1_{m,1},K^3_{m+1,1}],\ldots$$

By the definition of $A(n)$ (the number of distinct squares in $\T[1,n]$) and $a(n)=A(n)-A(n-1)$, we have
$$a(n)=\#\{\omega\mid \omega\omega\triangleright\T[1,n],\
\omega\omega\not\!\prec\T[1,n-1]\}.$$
By the definition of $K^j_{m,1}$ ($j=1,2,3$),  $a(n)=1$ if $n\in\cup_{m\geq4,j=1,2,3}K^j_{m,1}$, otherwise, $a(n)=0$. This achieves the proof of Proposition a.

Since $\A_m=\min K^1_{m,1}$, $\beta_m=\max K^3_{m+1,1}$, $\gamma_m=\min K^2_{m+1,1}$ and $\theta_m=\max K^2_{m+1,1}$,
by $A(n)=\sum_{i=1}^n a(i)$, Theorem A is an immediate consequence of Proposition a.

\subsection{Proofs of Proposition b. and Algorithm B}

Now we turn to give the number of repeated squares in $\T[1,n]$.

For $m\geq4$ and $p\geq1$, we consider the sets
\begin{equation}\label{E6}
\begin{cases}
\Gamma^1_{m,p}=[0..t_{m-2}-1]\oplus\big((K_m)_p+t_{m-3}+t_{m-4}\big);\\
\Gamma^2_{m,p}=[0..t_{m-3}-1]\oplus\big((K_m)_p+t_{m-4}\big);\\
\Gamma^3_{m,p}=[0..t_{m-4}-1]\oplus(K_m)_p.
\end{cases}
\end{equation}
Obviously, $\max\Gamma^2_{m,p}+1=\min\Gamma^1_{m,p}$,
$\max\Gamma^3_{m,p}+1=\min\Gamma^2_{m,p}$.
Moreover $\#\Gamma^1_{m,p}=t_{m-2}$, $\#\Gamma^2_{m,p}=t_{m-3}$ and $\#\Gamma^3_{m,p}=t_{m-4}$.

By Expression $(\ref{E5})$, $K^1_{m,p}$ (resp. $K^2_{m,p}$, $K^3_{m,p}$) contains several maximal (resp. maximal, minimal) elements in set $\Gamma^1_{m,p}$ (resp. $\Gamma^2_{m,p}$, $\Gamma^3_{m,p}$).
Thus we get a one-to-one correspondence between $\Gamma^i_{m,p}$ and $K^i_{m,p}$, $i=1,2,3$.

\begin{lemma}[]\label{Lt}
$1.\ |\T[1,a_p]|_a=|\T[1,b_p]|_b=|\T[1,c_p]|_c=p$;

$2.\ |\T[1,b_p]|_a=a_p$, $|\T[1,a_p|_b=|\T[1,p-1]|_a$;

$3.\ |\T[1,c_p]|_a=b_p$, $|\T[1,c_p]|_b=a_p$.
\end{lemma}

\begin{proof} 1. It holds by the definitions of $|\T[1,p]|_\gamma$ and $\gamma_p$, $\gamma\in\{a,b,c\}$.

2. By Properties \ref{Gt} and \ref{wpt}, the $p$-th occurrence of $aca=aG_2(a)a$ correspond to the $b_p$-th occurrence of letter $a$.
Moreover, the $(b_p+1)$-th letter $a$ occurs at position $(aca)_p$.
So $(aca)_p=a_{b_p+1}$.

Since $\K(aca)=c$,
$$(aca)_p=c_p+1=4p+3|\T[1,p-1]|_a+2|\T[1,p-1]|_b+1.$$

On the other hand, $b_p=2p+2|\T[1,p-1]|_a+|\T[1,p-1]|_b$, so
\begin{equation*}
\begin{split}
&a_{b_p+1}=b_p+1+|\T[1,b_p]|_a+|\T[1,b_p]|_b\\
=&3p+2|\T[1,p-1]|_a+|\T[1,p-1]|_b+|\T[1,b_p]|_a+1.
\end{split}
\end{equation*}

By $(aca)_p=a_{b_p+1}$, $|\T[1,b_p]|_a=p+|\T[1,p-1]|_a+|\T[1,p-1]|_b=a_p$.

Similarly, since $aba=aG_1(a)a$, the $(a_p+1)$-th letter $a$ occurs at position $(aba)_p$. This means $a_{a_p+1}=(aba)_p$.
By Property \ref{P}, the second equation holds.

3. By analogous arguments, $(aa)_p=a_{c_p+1}$ and $(bab)_p=b_{c_p+1}$.
Since $aa=K_4$ and $bab=K_5$, by Property \ref{P}, we get two equations
\begin{equation*}
\begin{cases}
|\T[1,c_p]|_a+|\T[1,c_p]|_b
=3p+3|\T[1,p-1]|_a+2|\T[1,p-1]|_b;\\
2|\T[1,c_p]|_a+|\T[1,c_p]|_b=5p+5|\T[1,p-1]|_a+3|\T[1,p-1]|_b.
\end{cases}
\end{equation*}
Thus $|\T[1,c_p]|_a=b_p$ and $|\T[1,c_p]|_b=a_p$. The conclusions hold.
\end{proof}

Using Lemma \ref{Lt},
comparing the minimal and maximal elements in these sets below, we have
\begin{equation}\label{E7}
\begin{cases}
\Gamma^1_{m,p}=[\Gamma^3_{m-1,a_p+1},\Gamma^2_{m-1,a_p+1},\Gamma^1_{m-1,a_p+1}],
&m\geq5;\\
\Gamma^2_{m,p}=[\Gamma^3_{m-2,b_p+1},\Gamma^2_{m-2,b_p+1},\Gamma^1_{m-2,b_p+1}],
&m\geq6;\\
\Gamma^3_{m,p}=[\Gamma^3_{m-3,c_p+1},\Gamma^2_{m-3,c_p+1},\Gamma^1_{m-3,c_p+1}],
&m\geq7.
\end{cases}
\end{equation}
Thus we establish recursive relations for any $\Gamma^i_{m,p}$,
$i\in\{1,2,3\}$, $m\geq4$, $p\geq1$.
By the one-to-one correspondence between $\Gamma^i_{m,p}$ and $K^i_{m,p}$, we can define a recursive structure over $\{K^i_{m,p}\}$.

The recursive structure is a directed graph $\mathcal{G}=(V,E)$ where:
\begin{equation*}
\begin{split}
V&=\{nodes\}=\{K^i_{m,p}\mid i=1,2,3,\ m\geq4,\ p\geq1\};\\
E&=\{edges\}=
\begin{cases}
K^1_{m+1,p} \rightarrow K^i_{m,a_p+1};\\
K^2_{m+2,p} \rightarrow K^i_{m,b_p+1};\\
K^3_{m+3,p} \rightarrow K^i_{m,c_p+1}.
\end{cases}
(i=1,2,3)
\end{split}
\end{equation*}
Here the notation ``$x\rightarrow y$'' means
a directed edge from the node $x$ to~$y$.

\begin{property}[]\label{P6.5}\
The recursive structure $\mathcal{G}$ is a family of finite rooted trees with
nodes
$\{K^i_{m,p}\mid i=1,2,3,\ m\geq4,\ p\geq1\}$ satisfying the following conditions:

$1.$ The roots are $\{K^i_{m,1}\mid i=1,2,3,\ m\geq4\}$;

$2.$ The leaves are
$\{K^1_{4,p},K^2_{4,p},K^3_{4,p},K^2_{5,p},K^3_{5,p},K^3_{6,p}\mid p\geq 1\}$.
\end{property}

\begin{proof}
Since $\mathbb{Z}^{+}=\{1\}\cup\{a_p+1\}\cup\{b_p+1\}\cup\{c_p+1\}$, nodes $\{K^i_{m,p}\mid i=1,2,3,\ m\geq4,\ p\geq1\}$ can be divide into four types:
\begin{equation*}
\begin{cases}
K^i_{m,1}\text{ has no ``parent'', i.e., it is a root};\\
K^i_{m,a_p+1}\text{ has a unique ``parent'' }K^1_{m+1,p};\\
K^i_{m,b_p+1}\text{ has a unique ``parent'' }K^2_{m+2,p};\\
K^i_{m,c_p+1}\text{ has a unique ``parent'' }K^3_{m+3,p}.
\end{cases}
\end{equation*}
Thus each node $K^i_{m,p}$ ($p\geq2$) has a unique ``parent'', so the recursive structure $\mathcal{G}$ is a family of finite trees, and only $K^i_{m,1}$ can be the root of these trees.
Furthermore, only the notes in condition 2 have no ``son". Thus they are leaves.
\end{proof}

By Equation (\ref{E7}) and the one-to-one correspondence between $\Gamma^i_{m,p}$ and $K^i_{m,p}$,
all integer elements in the nodes of the tree with root $K^i_{m,1}$ are elements of set $\Gamma^i_{m,1}$.
Since $\Gamma^i_{m,1}$'s are pairwise disjoint, the trees are pairwise disjoint too.
Figure \ref{Fig:1} shows three finite trees in the recursive structure $\mathcal{G}$ with roots $K^i_{8,1}$, $i=1,2,3$, respectively.

Since all integer elements in the nodes of the trees are the ending positions of
squares in the Tribonacci sequence $\T$, we call the recursive structure $\mathcal{G}$ ``\emph{square trees}".
Now me illustrate the relation between ``square trees" and ``squares in $\T$" with the help of an example.

\smallskip

\noindent\emph{Example.}  There are four integer 132's in the tree with root $K^2_{8,1}$, see Figure \ref{Fig:1}. Since all square trees are pairwise disjoint, so the integer 132 only exist in the tree with root $K^2_{8,1}$. By the definition of $K^i_{m,p}$, there are four squares ending at position 132, $\#\{\w\w\mid \w\w\triangleright\T[1,132]\}=4$.

\smallskip

\begin{figure}[!ht]
\setlength{\unitlength}{0.67mm}
\newsavebox{\B}
\savebox{\B}(10.5,4){
\begin{picture}(11,4)
\put(-0.5,-1){\line(1,0){11}}
\put(-0.5,3){\line(1,0){11}}
\put(-0.5,-1){\line(0,1){4}}
\put(10.5,-1){\line(0,1){4}}
\end{picture}}
\newsavebox{\BB}
\savebox{\BB}(9,7){
\begin{picture}(9,7)
\put(-1,-1){\line(1,0){9}}
\put(-1,6){\line(1,0){9}}
\put(-1,-1){\line(0,1){7}}
\put(8,-1){\line(0,1){7}}
\end{picture}}
\newsavebox{\BBB}
\savebox{\BBB}(9,13){
\begin{picture}(9,13)
\put(-1,-1){\line(1,0){9}}
\put(-1,9){\line(1,0){9}}
\put(-1,-1){\line(0,1){10}}
\put(8,-1){\line(0,1){10}}
\put(3,3.5){$\cdot$}
\put(3,2.5){$\cdot$}
\end{picture}}
\newsavebox{\D}
\savebox{\D}(6,2){
\begin{picture}(6,2)
\put(2,-0.5){$\cdot$}
\put(2,0.5){$\cdot$}
\end{picture}}

\centering
\tiny
\begin{picture}(168,175)
\put(0,35){$\mathbf{K^1_{8,1}}$}
\put(1,30){162}
\put(1,27){163}
\put(4,18){$\vdots$}
\put(4,12){$\vdots$}
\put(1,3){175}
\put(1,0){176}
\put(0,-1){\line(1,0){9}}
\put(0,33){\line(1,0){9}}
\put(0,-1){\line(0,1){34}}
\put(9,-1){\line(0,1){34}}
\put(30,20){$\mathbf{K^1_{7,2}}$}
\put(31,15){169}
\put(31,12){170}
\put(31,8){\usebox{\D}}
\put(31,3){175}
\put(31,0){176}
\put(30,-1){\line(1,0){9}}
\put(30,18){\line(1,0){9}}
\put(30,-1){\line(0,1){19}}
\put(39,-1){\line(0,1){19}}
\put(40,68){$\mathbf{K^2_{7,2}}$}
\put(41,63){145}
\put(41,60){146}
\put(41,57){\usebox{\D}}
\put(41,54){151}
\put(41,51){152}
\put(40,50){\line(1,0){9}}
\put(40,66){\line(1,0){9}}
\put(40,50){\line(0,1){16}}
\put(49,50){\line(0,1){16}}
\put(10,125){$\mathbf{K^2_{8,1}}$}
\put(11,120){118}
\put(11,117){119}
\put(14,108){$\vdots$}
\put(14,102){$\vdots$}
\put(11,93){131}
\put(11,90){132}
\put(10,89){\line(1,0){9}}
\put(10,123){\line(1,0){9}}
\put(10,89){\line(0,1){34}}
\put(19,89){\line(0,1){34}}
\put(20,170){$\mathbf{K^3_{8,1}}$}
\put(22,165){96}
\put(22,162){97}
\put(24,154){$\vdots$}
\put(21,147){106}
\put(20,146){\line(1,0){9}}
\put(20,168){\line(1,0){9}}
\put(20,146){\line(0,1){22}}
\put(29,146){\line(0,1){22}}
\put(50,89){$\mathbf{K^3_{7,2}}$}
\put(51,84){133}
\put(51,78){138}
\put(50,78){\usebox{\BBB}}
\put(62,101){$\mathbf{K^1_{6,3}}$}
\put(63,96){129}
\put(63,90){132}
\put(62,90){\usebox{\BBB}}
\put(62,11){$\mathbf{K^1_{6,4}}$}
\put(63,6){173}
\put(63,0){176}
\put(62,0){\usebox{\BBB}}
\put(72,128){$\mathbf{K^2_{6,3}}$}
\put(73,123){116}
\put(73,117){119}
\put(72,117){\usebox{\BBB}}
\put(72,38){$\mathbf{K^2_{6,4}}$}
\put(73,33){160}
\put(73,27){163}
\put(72,27){\usebox{\BBB}}
\put(82,143){$\mathbf{K^3_{6,3}}$}
\put(83,138){109}
\put(83,132){112}
\put(82,132){\usebox{\BBB}}
\put(82,53){$\mathbf{K^3_{6,4}}$}
\put(83,48){153}
\put(83,42){156}
\put(82,42){\usebox{\BBB}}
\put(94,149){$\mathbf{K^1_{5,5}}$}
\put(95,144){107}
\put(95,141){108}
\put(94,141){\usebox{\BB}}
\put(94,98){$\mathbf{K^1_{5,6}}$}
\put(95,93){131}
\put(95,90){132}
\put(94,90){\usebox{\BB}}
\put(94,59){$\mathbf{K^1_{5,7}}$}
\put(95,54){151}
\put(95,51){152}
\put(94,51){\usebox{\BB}}
\put(94,8){$\mathbf{K^1_{5,8}}$}
\put(95,3){175}
\put(95,0){176}
\put(94,0){\usebox{\BB}}
\put(104,164){$\mathbf{K^2_{5,5}}$}
\put(105,159){100}
\put(105,156){101}
\put(104,156){\usebox{\BB}}
\put(104,113){$\mathbf{K^2_{5,6}}$}
\put(105,108){124}
\put(105,105){125}
\put(104,105){\usebox{\BB}}
\put(104,71){$\mathbf{K^2_{5,7}}$}
\put(105,66){144}
\put(105,63){145}
\put(104,63){\usebox{\BB}}
\put(104,23){$\mathbf{K^2_{5,8}}$}
\put(105,18){168}
\put(105,15){169}
\put(104,15){\usebox{\BB}}
\put(116,170){$\mathbf{K^3_{5,5}}$}
\put(118,165){96}
\put(118,162){97}
\put(116,162){\usebox{\BB}}
\put(116,119){$\mathbf{K^3_{5,6}}$}
\put(117,114){120}
\put(117,111){121}
\put(116,111){\usebox{\BB}}
\put(116,77){$\mathbf{K^3_{5,7}}$}
\put(117,72){140}
\put(117,69){141}
\put(116,69){\usebox{\BB}}
\put(116,29){$\mathbf{K^3_{5,8}}$}
\put(117,24){164}
\put(117,21){165}
\put(116,21){\usebox{\BB}}
\put(128,146){$\mathbf{K^1_{4,9}}$}
\put(130,141){108}
\put(128,141){\usebox{\B}}
\put(128,122){$\mathbf{K^1_{4,10}}$}
\put(130,117){119}
\put(128,117){\usebox{\B}}
\put(128,95){$\mathbf{K^1_{4,11}}$}
\put(130,90){132}
\put(128,90){\usebox{\B}}
\put(128,80){$\mathbf{K^1_{4,12}}$}
\put(130,75){139}
\put(128,75){\usebox{\B}}
\put(128,56){$\mathbf{K^1_{4,13}}$}
\put(130,51){152}
\put(128,51){\usebox{\B}}
\put(128,32){$\mathbf{K^1_{4,14}}$}
\put(130,27){163}
\put(128,27){\usebox{\B}}
\put(128,5){$\mathbf{K^1_{4,15}}$}
\put(130,0){176}
\put(128,0){\usebox{\B}}
\put(142,155){$\mathbf{K^2_{4,9}}$}
\put(144,150){104}
\put(142,150){\usebox{\B}}
\put(142,131){$\mathbf{K^2_{4,10}}$}
\put(144,126){115}
\put(142,126){\usebox{\B}}
\put(142,104){$\mathbf{K^2_{4,11}}$}
\put(144,99){128}
\put(142,99){\usebox{\B}}
\put(142,86){$\mathbf{K^2_{4,12}}$}
\put(144,81){135}
\put(142,81){\usebox{\B}}
\put(142,62){$\mathbf{K^2_{4,13}}$}
\put(144,57){148}
\put(142,57){\usebox{\B}}
\put(142,41){$\mathbf{K^2_{4,14}}$}
\put(144,36){159}
\put(142,36){\usebox{\B}}
\put(142,11){$\mathbf{K^2_{4,15}}$}
\put(144,6){172}
\put(142,6){\usebox{\B}}
\put(156,158){$\mathbf{K^3_{4,9}}$}
\put(158,153){102}
\put(156,153){\usebox{\B}}
\put(156,134){$\mathbf{K^3_{4,10}}$}
\put(158,129){113}
\put(156,129){\usebox{\B}}
\put(156,107){$\mathbf{K^3_{4,11}}$}
\put(158,102){126}
\put(156,102){\usebox{\B}}
\put(156,89){$\mathbf{K^3_{4,12}}$}
\put(158,84){133}
\put(156,84){\usebox{\B}}
\put(156,65){$\mathbf{K^3_{4,13}}$}
\put(158,60){146}
\put(156,60){\usebox{\B}}
\put(156,44){$\mathbf{K^3_{4,14}}$}
\put(158,39){157}
\put(156,39){\usebox{\B}}
\put(156,17){$\mathbf{K^3_{4,15}}$}
\put(158,12){170}
\put(156,12){\usebox{\B}}
\put(106,2){\line(1,0){20}}
\put(106,2){\line(1,1){14}}\put(120,16){\line(1,0){34}}
\put(106,2){\line(2,1){14}}\put(120,9){\line(1,0){20}}
\put(106,53){\line(1,0){20}}
\put(106,53){\line(1,1){13}}\put(119,66){\line(1,0){35}}
\put(106,53){\line(2,1){14}}\put(120,60){\line(1,0){20}}
\put(106,92){\line(1,0){20}}
\put(106,92){\line(1,1){16}}\put(122,108){\line(1,0){32}}
\put(106,92){\line(2,1){14}}\put(120,99){\line(1,0){20}}
\put(106,143){\line(1,0){20}}
\put(106,143){\line(1,1){16}}\put(122,159){\line(1,0){32}}
\put(106,143){\line(2,1){14}}\put(120,150){\line(1,0){20}}
\put(83,33){\line(1,0){43}}
\put(83,33){\line(2,1){26}}\put(109,46){\line(1,0){45}}
\put(83,33){\line(3,1){21}}\put(104,40){\line(1,0){36}}
\put(83,123){\line(1,0){43}}
\put(83,123){\line(2,1){26}}\put(109,136){\line(1,0){45}}
\put(83,123){\line(3,1){21}}\put(104,130){\line(1,0){36}}
\put(73,2){\line(1,0){19}}
\put(73,2){\line(1,1){26}}\put(99,28){\line(1,0){15}}
\put(73,2){\line(2,1){28}}
\put(73,92){\line(1,0){19}}
\put(73,92){\line(1,1){26}}\put(99,118){\line(1,0){15}}
\put(73,92){\line(2,1){28}}
\put(61,84){\line(1,0){79}}
\put(61,84){\line(2,1){6}}\put(67,87){\line(1,0){74}}\put(154,87){\line(1,0){1}}
\put(61,84){\line(2,-1){6}}\put(67,81){\line(1,0){59}}
\put(51,58){\line(1,0){41}}
\put(51,58){\line(2,1){20}}\put(71,68){\line(1,0){31}}
\put(51,58){\line(1,1){18}}\put(69,76){\line(1,0){45}}
\put(41,2){\line(1,0){19}}
\put(41,2){\line(1,1){29}}
\put(41,2){\line(1,2){22}}\put(63,46){\line(1,0){17}}
\put(11,2){\line(1,0){17}}
\put(11,2){\line(1,2){27}}
\put(11,2){\line(1,3){26}}\put(37,80){\line(1,0){10}}
\put(21,96){\line(1,0){39}}
\put(21,96){\line(2,1){49}}
\put(21,96){\line(1,1){38}}\put(59,134){\line(1,0){21}}
\put(31,160){\line(1,0){71}}
\put(31,160){\line(2,1){20}}\put(51,170){\line(1,0){63}}
\put(31,160){\line(1,-1){10}}\put(41,150){\line(1,0){51}}
\end{picture}
\normalsize
\caption{Square trees with roots $K^i_{8,1}$, $i=1,2,3$.
The directed edges are always from left to right.
The number of squares ending at position $n$ is equal to the number of integer $n$ occurs in the square trees. For instance, $b(132)=4$ and $b(174)=3$.
\label{Fig:1}}
\end{figure}
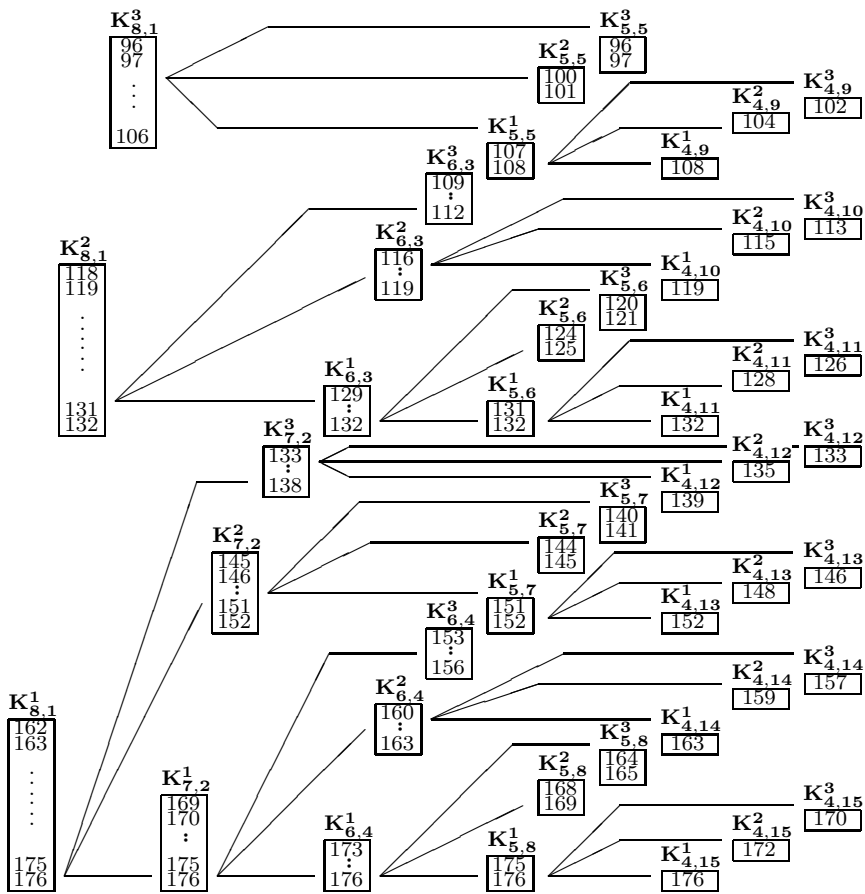

\begin{figure}[!ht]
\tiny
\setlength{\unitlength}{0.8mm}
\newsavebox{\C}
\savebox{\C}(10.5,4){
\begin{picture}(11,4)
\put(0.5,0){\line(1,0){8}}
\put(0.5,4){\line(1,0){8}}
\put(0.5,0){\line(0,1){4}}
\put(8.5,0){\line(0,1){4}}
\end{picture}}
\begin{center}
\begin{picture}(90,63)
\put(3,1){51}
\put(3,4){50}
\put(3,7){49}
\put(3,10){48}
\put(1,15){$\mathbf{K^1_{6,1}}$}
\put(1,0){\line(0,1){13}}
\put(1,0){\line(1,0){8}}
\put(9,0){\line(0,1){13}}
\put(1,13){\line(1,0){8}}
\put(10,10){\line(1,-1){7}}
\put(10,10){\line(1,2){9}}\put(19,28){\line(1,0){5}}
\put(25,27){the subtree with root $K^2_{5,2}$}
\put(10,10){\line(1,3){8}}\put(18,34){\line(1,0){6}}
\put(25,33){the subtree with root $K^3_{5,2}$}
\put(20,1){51}
\put(20,4){50}
\put(18,9){$\mathbf{K^1_{5,2}}$}
\put(18,0){\line(0,1){7}}
\put(18,0){\line(1,0){8}}
\put(26,0){\line(0,1){7}}
\put(18,7){\line(1,0){8}}
\put(20,41){27}
\put(20,44){26}
\put(18,49){$\mathbf{K^1_{5,1}}$}
\put(18,40){\line(0,1){7}}
\put(18,40){\line(1,0){8}}
\put(26,40){\line(0,1){7}}
\put(18,47){\line(1,0){8}}
\multiput(0,38)(3,0){32}{\line(1,0){2}}
\put(43,1){51}
\put(41,6){$\mathbf{K^1_{4,4}}$}
\put(40,0){{\usebox{\C}}}
\put(43,41){27}
\put(41,46){$\mathbf{K^1_{4,2}}$}
\put(40,40){{\usebox{\C}}}
\put(63,13){47}
\put(61,18){$\mathbf{K^2_{4,4}}$}
\put(60,12){{\usebox{\C}}}
\put(63,53){23}
\put(61,58){$\mathbf{K^2_{4,2}}$}
\put(60,52){{\usebox{\C}}}
\put(83,19){45}
\put(81,24){$\mathbf{K^3_{4,4}}$}
\put(80,18){{\usebox{\C}}}
\put(83,59){21}
\put(81,64){$\mathbf{K^3_{4,2}}$}
\put(80,58){{\usebox{\C}}}
\put(27,2){\line(1,0){13}}
\put(27,2){\line(1,1){12}}\put(39,14){\line(1,0){21}}
\put(27,2){\line(1,2){10}}\put(37,22){\line(1,0){43}}
\put(27,42){\line(1,0){13}}
\put(27,42){\line(1,1){12}}\put(39,54){\line(1,0){21}}
\put(27,42){\line(1,2){10}}\put(37,62){\line(1,0){43}}
\end{picture}
\end{center}
\normalsize
\caption{The graph embedding in square trees. For instance, the square tree with root $K^1_{5,1}$ and the subtree with root $K^1_{5,2}$ are isomorphic.}
\label{Fig:2}
\end{figure}
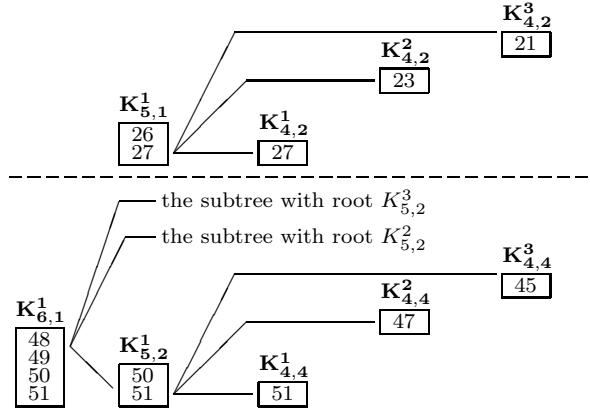

By the graph embedding in the recursive structure $\mathcal{G}$, see Figure~\ref{Fig:2}, we have Property \ref{P4.5}. This property helps us establish a recursive relation for counting the number of
repeated squares in $\T[1,n]$.

\begin{property}[]\label{P4.5}\ For $j=1,2,3$ and $1\leq i\leq t_{m-j-1}$
\begin{equation*}
\begin{split}
&\{\w\mid \w\w\triangleright\T[1,\Gamma^j_{m,1}[i]]\}\\
=&\{\w\mid \w\w\triangleright\T[1,\Gamma^j_{m,p}[i]],\K(\w)=K_h,1\leq h\leq m\}.
\end{split}
\end{equation*}
\end{property}

\noindent\emph{Example.}  Take $m=6$, $j=2$, $p=4$ and $i=7$ in Property \ref{P4.5}.
All squares ending at position $\Gamma^2_{6,1}[7]=38$ are $\{\nu\nu,\mu\mu\}$ where $\nu=\T[17,27]$ and $\mu=\T[25,31]$.
All squares ending at position
$\Gamma^2_{6,4}[7]=163$ are $\{\nu\nu,\mu\mu,\w\w\}$ where $\w=\T[2,82]$.
Since $\K(\nu\nu)=K_{6}$, $\K(\mu\mu)=K_4$, $\K(\w\w)=K_8$,
only $\{\nu\nu,\mu\mu\}$ are squares with kernel $K_{h}$, $1\leq h\leq 6$.

\begin{proof}[Proof of Proposition b.]
By the definitions of $\Gamma^i_{m,p}$ ($i=1,2,3$), $U_m$, $V_m$ and $W_m$,
using $(K_m)_1=\frac{t_{m}+t_{m-2}-1}{2}$, we have for $m\geq4$
\begin{equation*}
\begin{cases}
\Gamma^1_{m,1}=[1..t_{m-2}]\oplus\frac{3t_{m-1}+t_{m-3}-3}{2}=W_{m-1};\\
\Gamma^2_{m,1}=[1..t_{m-3}]\oplus\frac{3t_{m-1}-t_{m-3}-3}{2}=V_{m-1};\\
\Gamma^3_{m,1}=[1..t_{m-4}]\oplus\frac{t_{m}+t_{m-2}-3}{2}=U_{m-1}.
\end{cases}
\end{equation*}

By the definition of $B(n)$ and $b(n)=B(n)-B(n-1)$,
$b(n)$ means the number of squares ending at position $n$.
By the definition of $K^i_{m,p}$, the function $b(n)$ is equal to the number of integer $n$ occurs in the square trees.
By the graph embedding property in Property \ref{P4.5}, we can calculate $b(\Gamma^i_{m,1})$ ($i=1,2,3$) by recursive algorithm in Proposition~b.
\end{proof}

By $B(n)=\sum_{i=1}^n b(i)$, Algorithm B is an immediate consequence of Proposition b. So we prefer to omit the proof.

\section{The number of cubes}

By an analogous argument as Section 4, we count the number of distinct and repeated cubes in the Tribonacci sequence in this Section.
Let $\w$ be a factor with kernel $K_m$. By Proposition 6.7 in \cite{HW2015-2}, $\w\w\w\prec\T$ has only one case:
$G(K_m)=G_1$.
\begin{equation}\label{E8}
\begin{split}
\w\w\w
=&T_{m-1}[i,t_{m-1}]T_{m-2}T_{m-3}[1,t_{m-3}-1]\underline{K_{m+4}}\\
&T_{m+1}[k_{m+4},t_{m+1}]T_{m-3}T_{m-2}[1,i-1],
\end{split}
\end{equation}
where $1\leq i\leq \frac{t_{m-2}+t_{m-4}-1}{2}$, $m\geq3$, $\K(\w\w\w)=K_{m+4}$ and $|\w|=t_{m}$.

\medskip

\noindent\emph{Remark.}
By this case, we have that: all cubes in $\mathbb{T}$ are of length $3t_m$ for some $m\geq3$. Furthermore, for all $m\geq3$, there exists a cube of length $3t_m$ in $\mathbb{T}$. This is Theorem 7 in Mousavi-Shallit \cite{MS2014}.

\medskip

We consider the set for $m\geq7$ and $p\geq1$,
\begin{equation}\label{E9}
\begin{split}
K_{m,p}
=&\{(\w\w\w)_p\mid \K(\w\w\w)=K_m,|\w|=t_{m-4},\w\w\w\prec\mathbb{T}\}\\
=&[\tfrac{-t_{m-2}+5t_{m-4}+1}{2}
..t_{m-4}-1]\oplus(K_m)_p.
\end{split}
\end{equation}
Moreover $\#K_{m,p}=\frac{t_{m-6}+t_{m-8}-1}{2}$.

\subsection{Proofs of Proposition c. and Theorem C}

Take $p=1$ in Expression~(\ref{E9}), $K_{m,1}=[t_{m-1}+2t_{m-4}
..\tfrac{3t_{m-1}-t_{m-3}-3}{2}]$.
It is easy to check that $K_{m,1}$ are pairwise disjoint for different $m$. Thus we get a chain
$K_{7,1},K_{8,1}..
K_{m,1},\ldots$.
By the definition of $C(n)$ (the number of \emph{distinct cubes} in $\T[1,n]$)
and $c(n)=C(n)-C(n-1)$, we have
$$c(n)=\#\{\w\mid \w\w\w\triangleright\T[1,n],\
\w\w\w\not\!\prec\T[1,n-1]\}.$$
So $c(n)=1$ if $n\in\cup_{m\geq7} K_{m,1}$,
otherwise, $c(n)=0$.
This achieves the proof of Proposition c.
By $C(n)=\sum_{i=1}^n c(i)$, Theorem C is an immediate consequence of Proposition c.

\subsection{Proofs of Proposition d. and Algorithm D}

Now we turn to give the number of repeated cubes in $\T[1,n]$.
For $m\geq7$ and $p\geq1$, we consider the sets
\begin{equation}\label{E11}
\Gamma_{m,p}=[\Gamma^3_{m,p},\Gamma^2_{m,p},\Gamma^1_{m,p}]=[0..t_{m-1}-1]\oplus (K_m)_p.
\end{equation}
By Expressions (\ref{E9}) and (\ref{E11}),
we get a one-to-one correspondence between $\Gamma_{m,p}$ and $K_{m,p}$.
Using Lemma \ref{Lt},
comparing minimal and maximal elements in these sets, we have
$$\Gamma_{m,p}=[\Gamma_{m-3,c_p+1},\Gamma_{m-2,b_p+1},\Gamma_{m-1,a_p+1}]\text{ for }m\geq10.$$
So we can define a recursive structure over $\{K_{m,p}\mid m\geq7,p\geq1\}$.
The recursive structure is a directed graph $\mathcal{G}'=(V',E')$ where
\begin{equation*}
\begin{split}
V'&=\{nodes\}=\{K_{m,p}\mid m\geq7,p\geq1\};\\
E'&=\{edges\}=
\begin{cases}
K_{m+1,p} \rightarrow K_{m,a_p+1};\\
K_{m+2,p} \rightarrow K_{m,b_p+1};\\
K_{m+3,p} \rightarrow K_{m,c_p+1}.
\end{cases}
\end{split}
\end{equation*}
Here the notation ``$x\rightarrow y$'' means
a directed edge from the node $x$ to~$y$.

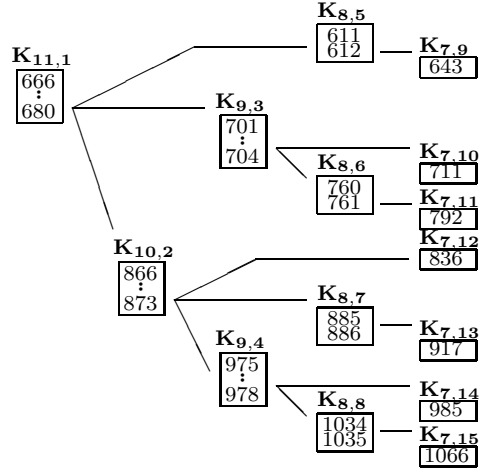
\begin{figure}[!ht]
\setlength{\unitlength}{0.67mm}
\newsavebox{\Q}
\savebox{\Q}(10.5,4){
\begin{picture}(11,4)
\put(-0.5,-1){\line(1,0){11}}
\put(-0.5,3){\line(1,0){11}}
\put(-0.5,-1){\line(0,1){4}}
\put(10.5,-1){\line(0,1){4}}
\end{picture}}
\newsavebox{\QQ}
\savebox{\QQ}(11,8){
\begin{picture}(11,8)
\put(-1,-1){\line(1,0){11}}
\put(-1,6.5){\line(1,0){11}}
\put(-1,-1){\line(0,1){7.5}}
\put(10,-1){\line(0,1){7.5}}
\end{picture}}
\newsavebox{\QQQ}
\savebox{\QQQ}(9,14){
\begin{picture}(9,14)
\put(-1,-1){\line(1,0){9}}
\put(-1,9.5){\line(1,0){9}}
\put(-1,-1){\line(0,1){10.5}}
\put(8,-1){\line(0,1){10.5}}
\put(2.5,3.5){$\cdot$}
\put(2.5,2.5){$\cdot$}
\end{picture}}

\centering
\tiny
\begin{picture}(95,92)
\put(1,68){\usebox{\QQQ}}
\put(2,68){680}
\put(2,74){666}
\put(0,79){$\mathbf{K_{11,1}}$}
\put(12,70){\line(1,0){26}}
\put(12,70){\line(1,-3){8}}
\put(12,70){\line(2,1){24}}\put(36,82){\line(1,0){22}}
\put(21,30){\usebox{\QQQ}}
\put(22,30){873}
\put(22,36){866}
\put(20,41){$\mathbf{K_{10,2}}$}
\put(32,32){\line(1,0){26}}
\put(32,32){\line(1,-2){7}}
\put(32,32){\line(2,1){16}}\put(48,40){\line(1,0){30}}
\put(41,12){\usebox{\QQQ}}
\put(42,12){978}
\put(42,18){975}
\put(40,23){$\mathbf{K_{9,4}}$}
\put(52,15){\line(1,0){26}}
\put(52,15){\line(1,-1){6}}
\put(41,59){\usebox{\QQQ}}
\put(42,59){704}
\put(42,65){701}
\put(40,70){$\mathbf{K_{9,3}}$}
\put(52,62){\line(1,0){26}}
\put(52,62){\line(1,-1){6}}
\put(60,3){\usebox{\QQ}}
\put(61,3){1035}
\put(61,6){1034}
\put(60,11){$\mathbf{K_{8,8}}$}
\put(60,24){\usebox{\QQ}}
\put(62,24){886}
\put(62,27){885}
\put(60,32){$\mathbf{K_{8,7}}$}
\put(60,50){\usebox{\QQ}}
\put(62,50){761}
\put(62,53){760}
\put(60,58){$\mathbf{K_{8,6}}$}
\put(60,80){\usebox{\QQ}}
\put(62,80){612}
\put(62,83){611}
\put(60,88){$\mathbf{K_{8,5}}$}
\put(80,0){\usebox{\Q}}
\put(81,0){1066}
\put(80,4){$\mathbf{K_{7,15}}$}
\put(80,9){\usebox{\Q}}
\put(82,9){985}
\put(80,13){$\mathbf{K_{7,14}}$}
\put(80,21){\usebox{\Q}}
\put(82,21){917}
\put(80,25){$\mathbf{K_{7,13}}$}
\put(80,39){\usebox{\Q}}
\put(82,39){836}
\put(80,43){$\mathbf{K_{7,12}}$}
\put(80,47){\usebox{\Q}}
\put(82,47){792}
\put(80,51){$\mathbf{K_{7,11}}$}
\put(80,56){\usebox{\Q}}
\put(82,56){711}
\put(80,60){$\mathbf{K_{7,10}}$}
\put(80,77){\usebox{\Q}}
\put(82,77){643}
\put(80,81){$\mathbf{K_{7,9}}$}
\put(72.5,81){\line(1,0){6}}
\put(72.5,51){\line(1,0){6}}
\put(72.5,27){\line(1,0){6}}
\put(72.5,6){\line(1,0){6}}
\end{picture}
\normalsize
\caption{Cube tree with root $K_{11,1}$.
The directed edges are always from left to right.
The number of cubes ending at position $n$ is equal to the number of integer $n$ occurs in the cube trees.
\label{Fig:4}}
\end{figure}

\begin{property}[]
The recursive structure $\mathcal{G}'$ is a family of finite rooted trees with
nodes $\{K_{m,p}\mid m\geq7,p\geq1\}$ satisfying the conditions:

$1.$ The roots are $\{K_{m,1}\mid m\geq7\}$;

$2.$ The leaves are $\{K_{7,p}\mid p\geq 1\}$.
\end{property}

By an analogous argument as Section 4, we call the recursive structure $\mathcal{G}'$ ``\emph{cube trees}". It has a similar graph embedding property as square trees.
This property helps us establish a recursive relation for counting the number of
repeated cubes in $\T[1,n]$, i.e., Proposition d. and Algorithm D.

\section{The number of $\A$-powers}

There are many other topics dealing with repetitions in words, such as \emph{fractional power}.
For instance, $\T[12,16]=ababa$ is a $\frac{5}{2}$-power. Recall the definition of $\A$-power in Section 1. The notion $\A$-power is a generalization of square (2-power) and cube (3-power).

For $\A\geq2$, $P\in\mathbb{Z}^{+}$ and $\w\prec\T$,
if $\T[P-\A|\w|+1,P]$ is an $\A$-power, $\T[P-2|\w|+1,P]$ is a square.
Thus we can determine all ending positions of $\A$-powers ($\A\geq2$) from the ending positions of squares.
Furthermore, if both $\T[P-2n..P-1]$ and $\T[P-2n+1..P]$ are squares,
$\T[P-2n,P]$ is obviously a $\left(2+\frac{1}{n}\right)$-power in $\T$.
If $[P-h..P]$ contains serial ending positions of squares of length $2n$,
set $[P-h+i..P]$ contains serial ending positions of $\left(2+\frac{i}{n}\right)$-powers of length $2n+i$, where $0\leq i\leq h$.

\medskip

\textbf{Case 1.} By Definition (\ref{N1}), the squares with ending positions in $K^1_{m,p}$ and $K^3_{m+3,q}$ have length $2t_{m-1}$.
By Expression (\ref{E5}) and Lemma~\ref{Lt},
set $V=[K^1_{m,c_p},K^3_{m+3,p},K^1_{m,c_p+1}]$
contains $\frac{t_{m+1}-t_{m-1}-1}{2}$'s ending positions of squares of length $2t_{m-1}$.
By the analysis in the preceding paragraph,
$$V[i,\tfrac{t_{m+1}-t_{m-1}-1}{2}]=[i..\tfrac{t_{m+1}-t_{m-1}-1}{2}]\oplus\big((K_{m+3})_p-\tfrac{t_{m+1}-t_{m-1}+1}{2}\big)$$
contains serial ending positions of $\left(2+\frac{i-1}{t_{m-1}}\right)$-powers of length $2t_{m-1}+i-1$, where $1\leq i\leq \frac{t_{m+1}-t_{m-1}-1}{2}$.

\smallskip

\noindent\emph{Example.} Taking $m=5$, $[K^1_{5,4},K^3_{8,1},K^1_{5,5}]=[94..108]$
contains 15's ending positions of $2$-powers (squares) of length $26$.
Thus $[95..108]$ contains 14's ending positions of $\left(2+\frac{1}{13}\right)$-powers of length $27$; $[107,108]$ contains two ending positions of $3$-powers (cubes) of length $39$;
And $[108]$ contains one ending position of $\left(3+\frac{1}{13}\right)$-power of length $40$.

\smallskip

\textbf{Case 2.} Consider other $K^j_{m,p}$'s for $j=1,2,3$, $m\geq4$ and $p\geq1$.

1. For $1\leq i\leq \frac{t_{m+1}-3t_{m-1}-1}{2}$ and $p\in \mathbb{Z}^{+}-\{c_q,c_q+1\mid q\geq1\}$,
\begin{equation*}
\begin{split}
&K^1_{m,p}[i,\tfrac{t_{m+1}-3t_{m-1}-1}{2}]\\
=~&[\tfrac{t_{m-1}-t_{m-3}+2i-1}{2}..t_{m-2}-1]\oplus\big((K_m)_p+t_{m-3}+t_{m-4}\big)
\end{split}
\end{equation*}
contains serial ending positions of $\left(2+\frac{i-1}{t_{m-1}}\right)$-powers.

2. For $1\leq i\leq \frac{t_{m+1}-3t_{m-1}-1}{2}$,
\begin{equation*}
\begin{split}
&K^2_{m,p}[i,\tfrac{t_{m+1}-3t_{m-1}-1}{2}]\\
=~&[\tfrac{t_{m-1}-2t_{m-2}+t_{m-3}+2i-1}{2}..t_{m-3}-1]\oplus\big((K_m)_p+t_{m-4}\big)
\end{split}
\end{equation*}
contains serial ending positions of $\left(2+\frac{i-1}{t_{m-3}+t_{m-4}}\right)$-powers.

3. For $1\leq i\leq \frac{-t_{m-2}+5t_{m-4}+1}{2}$ and $m\in\{4,5,6\}$,
\begin{equation*}
\begin{split}
K^3_{m,p}[i,\tfrac{-t_{m-2}+5t_{m-4}+1}{2}]
=[i-1..\tfrac{-t_{m-2}+5t_{m-4}-1}{2}]\oplus(K_m)_p
\end{split}
\end{equation*}
contains serial ending positions of $\left(2+\frac{i-1}{t_{m-4}}\right)$-powers.

\medskip

\noindent\emph{Remark.}
For a given sequence, it is an interesting and challenging task to determine its \emph{critical exponent} $e$, such that the sequence contains $\A$-powers for all $\A<e$, but has no $\A$-powers for
$\A>e$ (It may or may not have $e$-powers).
Let $\lambda$ be the only real zero of the polynomial $x^3-x^2-x-1=0$,
then $\lim_{m\rightarrow\infty}\tfrac{t_{m+1}}{t_{m}}=\lambda\approx1.84$,
see \cite{MS2014} for instance.

In Case 1, the maximal $\A$ is $\frac{3}{2}+\frac{t_{m+1}-3}{2t_{m-1}}$. So
$\lim_{m\rightarrow\infty}\max\A=\tfrac{3+\lambda^2}{2}$.

In Case 2, the maximal $\A$ are
$\tfrac{1}{2}+\tfrac{t_{m+1}-3}{2t_{m-1}},~
2+\tfrac{t_{m+1}-3t_{m-1}-3}{2t_{m-3}+2t_{m-4}},~\tfrac{9}{2}-\tfrac{t_{m-2}+1}{2t_{m-4}}$
for $j=1,2,3$ respectively. All of them are less than 3.
Thus there is no cube ending at these positions.

By the discussion above,
the critical exponent of $\T$ is $\tfrac{3+\lambda^2}{2}$.
This is different but equivalent to Theorem 4.5 in Tan-Wen \cite{TW2007}.

\medskip

Figure \ref{Fig:3} shows the evolving process of $\A$-power trees from square trees to cube trees.
In order to simplify the figure, we plot the ending positions in $\{K^j_{m,p}\mid j=1,2,3,\ p\geq1\}$ in one column.

Using the $\A$-power trees, we can count the numbers of
distinct and repeated $\A$-powers in $\T[1,n]$ for all $\A\geq2$ and $n\geq1$. But the expression is complicated. We prefer to omit it.

\begin{figure}[!ht]
\tiny
\setlength{\unitlength}{1mm}
\begin{center}
\begin{picture}(22,172)
\multiput(1.4,0)(0,2){77}{1}
\multiput(2.7,0)(0,2){7}{7}
\multiput(2.7,14)(0,2){10}{6}
\multiput(2.7,34)(0,2){10}{5}
\multiput(2.7,54)(0,2){10}{4}
\multiput(2.7,74)(0,2){10}{3}
\multiput(2.7,94)(0,2){10}{2}
\multiput(2.7,114)(0,2){10}{1}
\multiput(2.7,134)(0,2){10}{0}
\multiput(2.7,154)(0,2){4}{9}
\multiput(4,0)(0,20){9}{6}
\multiput(4,2)(0,20){8}{5}
\multiput(4,4)(0,20){8}{4}
\multiput(4,6)(0,20){8}{3}
\multiput(4,8)(0,20){8}{2}
\multiput(4,10)(0,20){8}{1}
\multiput(4,12)(0,20){8}{0}
\multiput(4,14)(0,20){8}{9}
\multiput(4,16)(0,20){8}{8}
\multiput(4,18)(0,20){8}{7}
\multiput(6.5,0)(3,0){6}{\line(0,1){166}}
\put(0,0){\line(0,1){166}}
\put(0,166){\line(1,0){21.5}}
\put(0,162){\line(1,0){21.5}}
\put(0,0){\line(1,0){21.5}}
\put(7,168){\small{$\alpha=2$}}
\put(3.4,164){$m$}
\put(1,162.6){$n$}
\put(0,165.5){\line(2,-1){6}}
\put(7.5,163){8}
\put(10.5,163){7}
\put(13.5,163){6}
\put(16.5,163){5}
\put(19.5,163){4}
\multiput(8,1)(0,2){15}{\circle*{2}}
\multiput(8,89)(0,2){15}{\circle*{2}}
\multiput(7.1,140)(0,2){11}{\hfill\rule{1.8mm}{1.8mm}}
\multiput(11,1)(0,2){8}{\circle*{2}}
\multiput(11,49)(0,2){8}{\circle*{2}}
\multiput(10.1,76)(0,2){6}{\hfill\rule{1.8mm}{1.8mm}}
\multiput(14,1)(0,2){4}{\circle*{2}}
\multiput(14,27)(0,2){4}{\circle*{2}}
\multiput(14,41)(0,2){4}{\circle*{2}}
\multiput(14,89)(0,2){4}{\circle*{2}}
\multiput(14,115)(0,2){4}{\circle*{2}}
\multiput(14,129)(0,2){4}{\circle*{2}}
\multiput(17,1)(0,2){2}{\circle*{2}}
\multiput(17,15)(0,2){2}{\circle*{2}}
\multiput(17,23)(0,2){2}{\circle*{2}}
\multiput(17,49)(0,2){2}{\circle*{2}}
\multiput(17,63)(0,2){2}{\circle*{2}}
\multiput(17,71)(0,2){2}{\circle*{2}}
\multiput(17,89)(0,2){2}{\circle*{2}}
\multiput(17,103)(0,2){2}{\circle*{2}}
\multiput(17,111)(0,2){2}{\circle*{2}}
\multiput(16.1,136)(0,2){2}{\hfill\rule{1.8mm}{1.8mm}}
\multiput(17,151)(0,2){2}{\circle*{2}}
\multiput(17,159)(0,2){2}{\circle*{2}}
\put(20,1){\circle*{2}}
\put(20,9){\circle*{2}}
\put(20,13){\circle*{2}}
\put(20,27){\circle*{2}}
\put(20,35){\circle*{2}}
\put(20,39){\circle*{2}}
\put(20,49){\circle*{2}}
\put(20,57){\circle*{2}}
\put(20,61){\circle*{2}}
\put(19.1,74){\hfill\rule{1.8mm}{1.8mm}}
\put(20,83){\circle*{2}}
\put(20,87){\circle*{2}}
\put(19.1,88){\hfill\rule{1.8mm}{1.8mm}}
\put(20,97){\circle*{2}}
\put(20,101){\circle*{2}}
\put(20,115){\circle*{2}}
\put(20,123){\circle*{2}}
\put(20,127){\circle*{2}}
\put(20,137){\circle*{2}}
\put(20,145){\circle*{2}}
\put(20,149){\circle*{2}}
\end{picture}
\begin{picture}(22,172)
\multiput(1.4,0)(0,2){77}{1}
\multiput(2.7,0)(0,2){7}{7}
\multiput(2.7,14)(0,2){10}{6}
\multiput(2.7,34)(0,2){10}{5}
\multiput(2.7,54)(0,2){10}{4}
\multiput(2.7,74)(0,2){10}{3}
\multiput(2.7,94)(0,2){10}{2}
\multiput(2.7,114)(0,2){10}{1}
\multiput(2.7,134)(0,2){10}{0}
\multiput(2.7,154)(0,2){4}{9}
\multiput(4,0)(0,20){9}{6}
\multiput(4,2)(0,20){8}{5}
\multiput(4,4)(0,20){8}{4}
\multiput(4,6)(0,20){8}{3}
\multiput(4,8)(0,20){8}{2}
\multiput(4,10)(0,20){8}{1}
\multiput(4,12)(0,20){8}{0}
\multiput(4,14)(0,20){8}{9}
\multiput(4,16)(0,20){8}{8}
\multiput(4,18)(0,20){8}{7}
\multiput(6.5,0)(3,0){6}{\line(0,1){166}}
\put(0,0){\line(0,1){166}}
\put(0,166){\line(1,0){21.5}}
\put(0,162){\line(1,0){21.5}}
\put(0,0){\line(1,0){21.5}}
\put(2,168){\small{$\alpha=2+\frac{1}{81}$}}
\put(3.4,164){$m$}
\put(1,162.6){$n$}
\put(0,165.5){\line(2,-1){6}}
\put(7.5,163){8}
\put(10.5,163){7}
\put(13.5,163){6}
\put(16.5,163){5}
\put(19.5,163){4}
\multiput(8,1)(0,2){15}{\circle{2}}
\multiput(8,1)(0,2){14}{\circle*{2}}
\multiput(8,89)(0,2){15}{\circle{2}}
\multiput(8,89)(0,2){14}{\circle*{2}}
\multiput(7.1,140)(0,2){11}{\hfill\rule{1.8mm}{1.8mm}}
\multiput(11,1)(0,2){8}{\circle{2}}
\multiput(11,1)(0,2){7}{\circle*{2}}
\multiput(11,49)(0,2){8}{\circle{2}}
\multiput(11,49)(0,2){7}{\circle*{2}}
\multiput(10.1,76)(0,2){6}{\hfill\rule{1.8mm}{1.8mm}}
\multiput(14,1)(0,2){4}{\circle{2}}
\multiput(14,27)(0,2){4}{\circle{2}}
\multiput(14,41)(0,2){4}{\circle{2}}
\multiput(14,89)(0,2){4}{\circle{2}}
\multiput(14,115)(0,2){4}{\circle{2}}
\multiput(14,129)(0,2){4}{\circle{2}}
\multiput(14,1)(0,2){3}{\circle*{2}}
\multiput(14,27)(0,2){3}{\circle*{2}}
\multiput(14,41)(0,2){3}{\circle*{2}}
\multiput(14,89)(0,2){3}{\circle*{2}}
\multiput(14,115)(0,2){3}{\circle*{2}}
\multiput(14,129)(0,2){3}{\circle*{2}}
\multiput(17,1)(0,2){2}{\circle{2}}
\multiput(17,15)(0,2){2}{\circle{2}}
\multiput(17,23)(0,2){2}{\circle{2}}
\multiput(17,49)(0,2){2}{\circle{2}}
\multiput(17,63)(0,2){2}{\circle{2}}
\multiput(17,71)(0,2){2}{\circle{2}}
\multiput(17,89)(0,2){2}{\circle{2}}
\multiput(17,103)(0,2){2}{\circle{2}}
\multiput(17,111)(0,2){2}{\circle{2}}
\put(17,1){\circle*{2}}
\put(17,15){\circle*{2}}
\put(17,23){\circle*{2}}
\put(17,49){\circle*{2}}
\put(17,63){\circle*{2}}
\put(17,71){\circle*{2}}
\put(17,89){\circle*{2}}
\put(17,103){\circle*{2}}
\put(17,111){\circle*{2}}
\multiput(16.1,136)(0,2){2}{\hfill\rule{1.8mm}{1.8mm}}
\multiput(17,151)(0,2){2}{\circle{2}}
\multiput(17,159)(0,2){2}{\circle{2}}
\put(17,151){\circle*{2}}
\put(17,159){\circle*{2}}
\put(20,1){\circle{2}}
\put(20,9){\circle{2}}
\put(20,13){\circle{2}}
\put(20,27){\circle{2}}
\put(20,35){\circle{2}}
\put(20,39){\circle{2}}
\put(20,49){\circle{2}}
\put(20,57){\circle{2}}
\put(20,61){\circle{2}}
\put(19.1,74){\hfill\rule{1.8mm}{1.8mm}}
\put(20,83){\circle{2}}
\put(20,87){\circle{2}}
\put(18.9,88){$\Box$}
\put(20,97){\circle{2}}
\put(20,101){\circle{2}}
\put(20,115){\circle{2}}
\put(20,123){\circle{2}}
\put(20,127){\circle{2}}
\put(20,137){\circle{2}}
\put(20,145){\circle{2}}
\put(20,149){\circle{2}}
\end{picture}
\begin{picture}(22,172)
\multiput(1.4,0)(0,2){77}{1}
\multiput(2.7,0)(0,2){7}{7}
\multiput(2.7,14)(0,2){10}{6}
\multiput(2.7,34)(0,2){10}{5}
\multiput(2.7,54)(0,2){10}{4}
\multiput(2.7,74)(0,2){10}{3}
\multiput(2.7,94)(0,2){10}{2}
\multiput(2.7,114)(0,2){10}{1}
\multiput(2.7,134)(0,2){10}{0}
\multiput(2.7,154)(0,2){4}{9}
\multiput(4,0)(0,20){9}{6}
\multiput(4,2)(0,20){8}{5}
\multiput(4,4)(0,20){8}{4}
\multiput(4,6)(0,20){8}{3}
\multiput(4,8)(0,20){8}{2}
\multiput(4,10)(0,20){8}{1}
\multiput(4,12)(0,20){8}{0}
\multiput(4,14)(0,20){8}{9}
\multiput(4,16)(0,20){8}{8}
\multiput(4,18)(0,20){8}{7}
\multiput(6.5,0)(3,0){6}{\line(0,1){166}}
\put(0,0){\line(0,1){166}}
\put(0,166){\line(1,0){21.5}}
\put(0,162){\line(1,0){21.5}}
\put(0,0){\line(1,0){21.5}}
\put(3,168){\small{$\alpha=2+\frac{1}{6}$}}
\put(3.4,164){$m$}
\put(1,162.6){$n$}
\put(0,165.5){\line(2,-1){6}}
\put(7.5,163){8}
\put(10.5,163){7}
\put(13.5,163){6}
\put(16.5,163){5}
\put(19.5,163){4}
\multiput(8,1)(0,2){15}{\circle{2}}
\multiput(8,1)(0,2){1}{\circle*{2}}
\multiput(8,89)(0,2){15}{\circle{2}}
\multiput(8,89)(0,2){8}{\circle*{2}}
\multiput(7.2,140.1)(0,2){10}{\hfill\rule{1.8mm}{1.8mm}}
\multiput(7,140)(0,2){11}{$\Box$}
\multiput(11,1)(0,2){8}{\circle{2}}
\multiput(11,49)(0,2){8}{\circle{2}}
\multiput(11,49)(0,2){3}{\circle*{2}}
\multiput(10.2,76.1)(0,2){5}{\hfill\rule{1.8mm}{1.8mm}}
\multiput(10,76)(0,2){6}{$\Box$}
\multiput(14,1)(0,2){4}{\circle{2}}
\multiput(14,27)(0,2){4}{\circle{2}}
\multiput(14,41)(0,2){4}{\circle{2}}
\multiput(14,89)(0,2){4}{\circle{2}}
\multiput(14,115)(0,2){4}{\circle{2}}
\multiput(14,129)(0,2){4}{\circle{2}}
\multiput(14,27)(0,2){2}{\circle*{2}}
\multiput(14,41)(0,2){3}{\circle*{2}}
\multiput(14,115)(0,2){2}{\circle*{2}}
\multiput(14,129)(0,2){3}{\circle*{2}}
\multiput(17,1)(0,2){2}{\circle{2}}
\multiput(17,15)(0,2){2}{\circle{2}}
\multiput(17,23)(0,2){2}{\circle{2}}
\multiput(17,49)(0,2){2}{\circle{2}}
\multiput(17,63)(0,2){2}{\circle{2}}
\multiput(17,71)(0,2){2}{\circle{2}}
\multiput(17,89)(0,2){2}{\circle{2}}
\multiput(17,103)(0,2){2}{\circle{2}}
\multiput(17,111)(0,2){2}{\circle{2}}
\put(17,15){\circle*{2}}
\put(17,23){\circle*{2}}
\put(17,63){\circle*{2}}
\put(17,71){\circle*{2}}
\put(17,103){\circle*{2}}
\put(17,111){\circle*{2}}
\multiput(16.1,136)(0,2){2}{\hfill\rule{1.8mm}{1.8mm}}
\multiput(17,151)(0,2){2}{\circle{2}}
\multiput(17,159)(0,2){2}{\circle{2}}
\put(17,151){\circle*{2}}
\put(17,159){\circle*{2}}
\put(20,1){\circle{2}}
\put(20,9){\circle{2}}
\put(20,13){\circle{2}}
\put(20,27){\circle{2}}
\put(20,35){\circle{2}}
\put(20,39){\circle{2}}
\put(20,49){\circle{2}}
\put(20,57){\circle{2}}
\put(20,61){\circle{2}}
\put(19.1,74){\hfill\rule{1.8mm}{1.8mm}}
\put(20,83){\circle{2}}
\put(20,87){\circle{2}}
\put(18.9,88){$\Box$}
\put(20,97){\circle{2}}
\put(20,101){\circle{2}}
\put(20,115){\circle{2}}
\put(20,123){\circle{2}}
\put(20,127){\circle{2}}
\put(20,137){\circle{2}}
\put(20,145){\circle{2}}
\put(20,149){\circle{2}}
\end{picture}
\begin{picture}(22,172)
\multiput(1.4,0)(0,2){77}{1}
\multiput(2.7,0)(0,2){7}{7}
\multiput(2.7,14)(0,2){10}{6}
\multiput(2.7,34)(0,2){10}{5}
\multiput(2.7,54)(0,2){10}{4}
\multiput(2.7,74)(0,2){10}{3}
\multiput(2.7,94)(0,2){10}{2}
\multiput(2.7,114)(0,2){10}{1}
\multiput(2.7,134)(0,2){10}{0}
\multiput(2.7,154)(0,2){4}{9}
\multiput(4,0)(0,20){9}{6}
\multiput(4,2)(0,20){8}{5}
\multiput(4,4)(0,20){8}{4}
\multiput(4,6)(0,20){8}{3}
\multiput(4,8)(0,20){8}{2}
\multiput(4,10)(0,20){8}{1}
\multiput(4,12)(0,20){8}{0}
\multiput(4,14)(0,20){8}{9}
\multiput(4,16)(0,20){8}{8}
\multiput(4,18)(0,20){8}{7}
\multiput(6.5,0)(3,0){6}{\line(0,1){166}}
\put(0,0){\line(0,1){166}}
\put(0,166){\line(1,0){21.5}}
\put(0,162){\line(1,0){21.5}}
\put(0,0){\line(1,0){21.5}}
\put(3,168){\small{$\alpha=2+\frac{1}{3}$}}
\put(3.4,164){$m$}
\put(1,162.6){$n$}
\put(0,165.5){\line(2,-1){6}}
\put(7.5,163){8}
\put(10.5,163){7}
\put(13.5,163){6}
\put(16.5,163){5}
\put(19.5,163){4}
\multiput(8,1)(0,2){15}{\circle{2}}
\multiput(8,89)(0,2){15}{\circle{2}}
\multiput(8,89)(0,2){2}{\circle*{2}}
\multiput(7.2,140.1)(0,2){8}{\hfill\rule{1.8mm}{1.8mm}}
\multiput(7,140)(0,2){11}{$\Box$}
\multiput(11,1)(0,2){8}{\circle{2}}
\multiput(11,49)(0,2){8}{\circle{2}}
\multiput(11,49)(0,2){1}{\circle*{2}}
\multiput(10.2,76.1)(0,2){4}{\hfill\rule{1.8mm}{1.8mm}}
\multiput(10,76)(0,2){6}{$\Box$}
\multiput(14,1)(0,2){4}{\circle{2}}
\multiput(14,27)(0,2){4}{\circle{2}}
\multiput(14,41)(0,2){4}{\circle{2}}
\multiput(14,89)(0,2){4}{\circle{2}}
\multiput(14,115)(0,2){4}{\circle{2}}
\multiput(14,129)(0,2){4}{\circle{2}}
\multiput(14,41)(0,2){2}{\circle*{2}}
\multiput(14,129)(0,2){2}{\circle*{2}}
\multiput(17,1)(0,2){2}{\circle{2}}
\multiput(17,15)(0,2){2}{\circle{2}}
\multiput(17,23)(0,2){2}{\circle{2}}
\multiput(17,49)(0,2){2}{\circle{2}}
\multiput(17,63)(0,2){2}{\circle{2}}
\multiput(17,71)(0,2){2}{\circle{2}}
\multiput(17,89)(0,2){2}{\circle{2}}
\multiput(17,103)(0,2){2}{\circle{2}}
\multiput(17,111)(0,2){2}{\circle{2}}
\put(17,23){\circle*{2}}
\put(17,71){\circle*{2}}
\put(17,111){\circle*{2}}
\multiput(16.1,136)(0,2){2}{\hfill\rule{1.8mm}{1.8mm}}
\multiput(17,151)(0,2){2}{\circle{2}}
\multiput(17,159)(0,2){2}{\circle{2}}
\put(17,159){\circle*{2}}
\put(20,1){\circle{2}}
\put(20,9){\circle{2}}
\put(20,13){\circle{2}}
\put(20,27){\circle{2}}
\put(20,35){\circle{2}}
\put(20,39){\circle{2}}
\put(20,49){\circle{2}}
\put(20,57){\circle{2}}
\put(20,61){\circle{2}}
\put(19.1,74){\hfill\rule{1.8mm}{1.8mm}}
\put(20,83){\circle{2}}
\put(20,87){\circle{2}}
\put(18.9,88){$\Box$}
\put(20,97){\circle{2}}
\put(20,101){\circle{2}}
\put(20,115){\circle{2}}
\put(20,123){\circle{2}}
\put(20,127){\circle{2}}
\put(20,137){\circle{2}}
\put(20,145){\circle{2}}
\put(20,149){\circle{2}}
\end{picture}
\begin{picture}(22,172)
\multiput(1.4,0)(0,2){77}{1}
\multiput(2.7,0)(0,2){7}{7}
\multiput(2.7,14)(0,2){10}{6}
\multiput(2.7,34)(0,2){10}{5}
\multiput(2.7,54)(0,2){10}{4}
\multiput(2.7,74)(0,2){10}{3}
\multiput(2.7,94)(0,2){10}{2}
\multiput(2.7,114)(0,2){10}{1}
\multiput(2.7,134)(0,2){10}{0}
\multiput(2.7,154)(0,2){4}{9}
\multiput(4,0)(0,20){9}{6}
\multiput(4,2)(0,20){8}{5}
\multiput(4,4)(0,20){8}{4}
\multiput(4,6)(0,20){8}{3}
\multiput(4,8)(0,20){8}{2}
\multiput(4,10)(0,20){8}{1}
\multiput(4,12)(0,20){8}{0}
\multiput(4,14)(0,20){8}{9}
\multiput(4,16)(0,20){8}{8}
\multiput(4,18)(0,20){8}{7}
\multiput(6.5,0)(3,0){6}{\line(0,1){166}}
\put(0,0){\line(0,1){166}}
\put(0,166){\line(1,0){21.5}}
\put(0,162){\line(1,0){21.5}}
\put(0,0){\line(1,0){21.5}}
\put(7,168){\small{$\alpha=3$}}
\put(3.4,164){$m$}
\put(1,162.6){$n$}
\put(0,165.5){\line(2,-1){6}}
\put(7.5,163){8}
\put(10.5,163){7}
\put(13.5,163){6}
\put(16.5,163){5}
\put(19.5,163){4}
\multiput(8,1)(0,2){15}{\circle{2}}
\multiput(8,89)(0,2){15}{\circle{2}}
\multiput(7,140)(0,2){11}{$\Box$}
\multiput(11,1)(0,2){8}{\circle{2}}
\multiput(11,49)(0,2){8}{\circle{2}}
\multiput(10.1,76)(0,2){6}{$\Box$}
\multiput(14,1)(0,2){4}{\circle{2}}
\multiput(14,27)(0,2){4}{\circle{2}}
\multiput(14,41)(0,2){4}{\circle{2}}
\multiput(14,89)(0,2){4}{\circle{2}}
\multiput(14,115)(0,2){4}{\circle{2}}
\multiput(14,129)(0,2){4}{\circle{2}}
\multiput(17,1)(0,2){2}{\circle{2}}
\multiput(17,15)(0,2){2}{\circle{2}}
\multiput(17,23)(0,2){2}{\circle{2}}
\multiput(17,49)(0,2){2}{\circle{2}}
\multiput(17,63)(0,2){2}{\circle{2}}
\multiput(17,71)(0,2){2}{\circle{2}}
\multiput(17,89)(0,2){2}{\circle{2}}
\multiput(17,103)(0,2){2}{\circle{2}}
\multiput(17,111)(0,2){2}{\circle{2}}
\multiput(16.1,136)(0,2){2}{\hfill\rule{1.8mm}{1.8mm}}
\multiput(17,151)(0,2){2}{\circle{2}}
\multiput(17,159)(0,2){2}{\circle{2}}
\put(20,1){\circle{2}}
\put(20,9){\circle{2}}
\put(20,13){\circle{2}}
\put(20,27){\circle{2}}
\put(20,35){\circle{2}}
\put(20,39){\circle{2}}
\put(20,49){\circle{2}}
\put(20,57){\circle{2}}
\put(20,61){\circle{2}}
\put(19.1,74){\hfill\rule{1.8mm}{1.8mm}}
\put(20,83){\circle{2}}
\put(20,87){\circle{2}}
\put(18.9,88){$\Box$}
\put(20,97){\circle{2}}
\put(20,101){\circle{2}}
\put(20,115){\circle{2}}
\put(20,123){\circle{2}}
\put(20,127){\circle{2}}
\put(20,137){\circle{2}}
\put(20,145){\circle{2}}
\put(20,149){\circle{2}}
\end{picture}
\end{center}
\normalsize
\caption{The evolving process of $\alpha$-power trees from square trees to cube trees.
Here we plot $K^j_{m,p}$ for different $j$ in one column. We use boxes (positions in Case 1) and circles (positions in Case 2)  to show the ending positions of all $\alpha$-powers. These boxes and circles appear filled (a $\alpha$-power ending at this position) or empty (a square ending at this position, but no $\alpha$-power ending here).}
\label{Fig:3}
\end{figure}
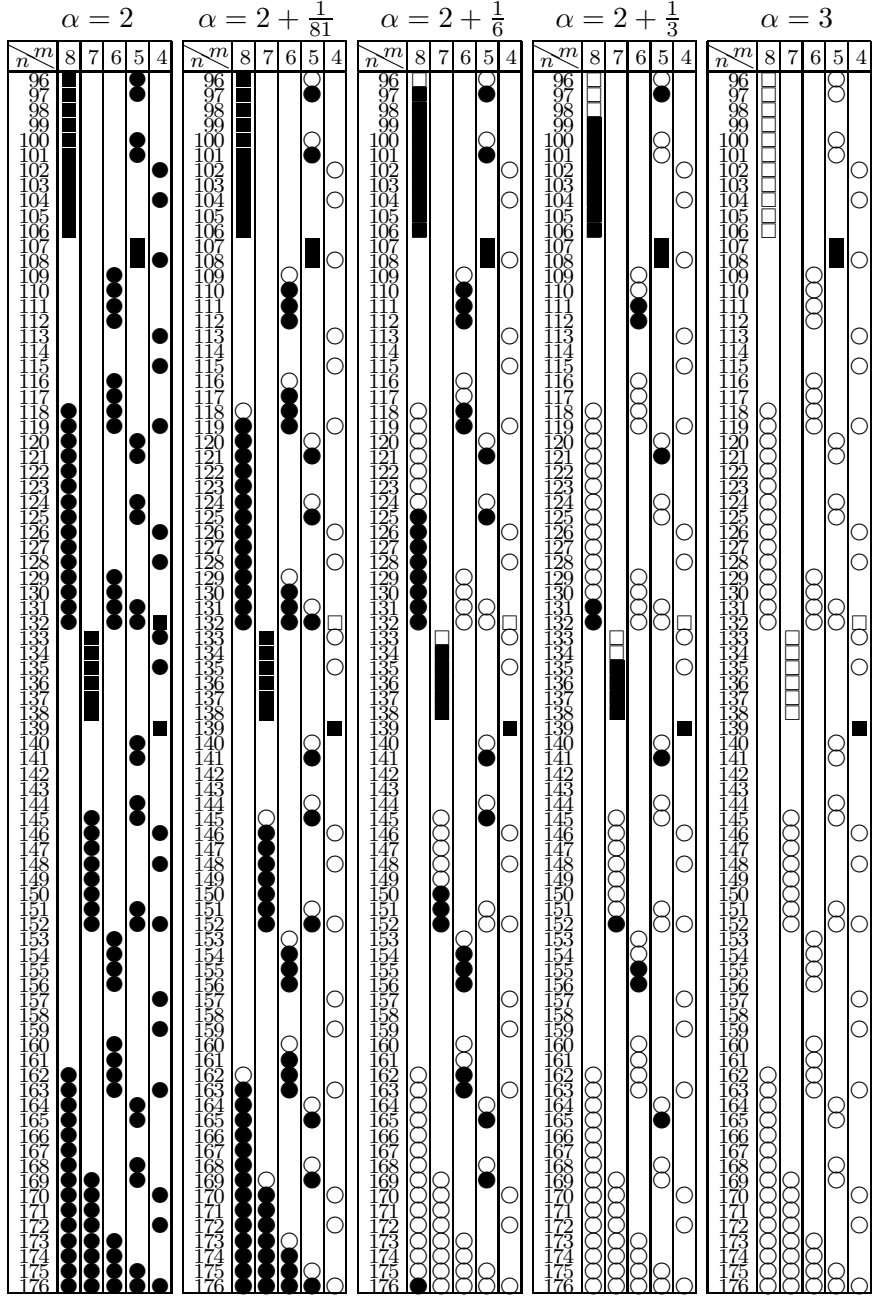


\begin{thebibliography}{AB}


\bibitem{AS2003} J. P. Allouche, J. Shallit, \emph{Automatic sequences. Theory, applications, generalizations}. Cambridge, 2003.

\bibitem{DL2003} D. Damanik, D. Lenz, Powers in Sturmian sequences, \emph{European Journal of Combinatorics}. 24.4 (2003) 377-390.

\bibitem{G2006} A. Glen, \emph{On Sturmian and Episturmian Words, and Related Topics,} Ph.D. thesis. The University of Adelaide, Australia, 2006.

\bibitem{FBFMS2002} N. P. Fogg, V. Berth\'e, S. Ferenczi, C. Mauduit, A. Siegel, \emph{Substitutions in Dynamics, Arithmetics and Combinatorics}. Springer, 2002.

\bibitem{FS1999} A. S. Fraenkel, J. Simpson, The exact number of squares in Fibonacci words, \emph{Theoretical Computer Science}. 218 (1999) 95-106.

\bibitem{FS2014} A. S. Fraenkel, J. Simpson, Corrigendum to ``The exact number of squares in Fibonacci words'', \emph{Theoretical Computer Science}. 547 (2014) 122.

\bibitem{HW2015-2} Y.-K. Huang, Z.-Y. Wen, Kernel words and gap sequence of the Tribonacci sequence, \emph{Acta Mathematica Scientia (Series B)}. 36.1 (2016) 173-194.

\bibitem{JP2002} J. Justin, G. Pirillo, On a characteristic property of Arnoux-Rauzy sequences, \emph{RAIRO-Theoretical Informatics and Applications}. 36.4 (2002) 385-388.

\bibitem{MS2014} H. Mousavi, J. Shallit, Mechanical proofs of properties of the Tribonacci word, in \emph{Combinatorics on Words}. Springer International Publishing. (2014) 170-190.

\bibitem{S2010} K. Saari, Everywhere $\A$-repetitive sequences and Sturmian words, \emph{European Journal of Combinatorics}. 31.1 (2010) 177-192.

\bibitem{TW2007} B. Tan, Z.-Y. Wen, Some properties of the Tribonacci sequence.
\emph{European Journal of Combinatorics}. 28 (2007) 1703-1719.

\end{thebibliography}
\end{document}